\documentclass[11pt]{amsart}
\usepackage{amsmath, amsthm, stmaryrd, latexsym, amscd, amsfonts, amssymb, mathrsfs, multicol, bm,}

\usepackage{enumerate}
\usepackage[all]{xy}
\usepackage{graphics,color}
\usepackage{todonotes}
\usepackage{boxedminipage}
\usepackage{epsfig}

\theoremstyle{plain}
\newtheorem*{mainthm}{Main Theorem}

\theoremstyle{plain}
\newtheorem{thm}{Theorem}[section]

\theoremstyle{definition}
\newtheorem{defn}[thm]{Definition}

\theoremstyle{plain}
\newtheorem{prop}[thm]{Proposition}

\theoremstyle{plain}
\newtheorem{lemma}[thm]{Lemma}

\theoremstyle{plain}
\newtheorem{cor}[thm]{Corollary}

\theoremstyle{remark}
\newtheorem{ex}[thm]{Example}

\theoremstyle{remark}
\newtheorem{rmk}[thm]{Remark}
\newtheorem{remark}[thm]{Remark}

\theoremstyle{remark}
\newtheorem{fact}[thm]{Fact}

\numberwithin{equation}{section}

\newcommand{\R}{\underline{R}} \newcommand{\Z}{\mathbb{Z}} \newcommand{\M}{\underline{M}} \newcommand{\su}{\underline{S}} 
\newcommand{\lu}{\underline{L}}
\newcommand{\tambg}{\mathcal{T}amb_G}
\newcommand{\tambh}{\mathcal{T}amb_H}
\newcommand{\mackh}{\mathcal{M}ack_H}
\newcommand{\mackg}{\mathcal{M}ack_G}
\newcommand{\greenh}{\mathcal{G}reen_H}
\newcommand{\greeng}{\mathcal{G}reen_G}
\newcommand{\bx}{\oblong}
\newcommand{\bfm}{{\bm{m}}}

\newcommand{\ba}{{\bm{a}}}
\newcommand{\bb}{\bm{b}}

\newcommand{\bmx}{{\bm{x}}}

\newcommand{\ngh}{N^G_H}

\newcommand{\ngk}{N^G_K}
\newcommand{\nngh}{\mathcal{N}^G_H}
\newcommand{\ih}{i_H^*}
\newcommand{\ik}{i_K^*}
\newcommand{\A}{\underline{A}}

\newcommand{\og}{\mathcal{O}_G}

\newcommand{\set}{\mathscr{S}et_G^{Fin}}

\newcommand{\g}{\gamma}

\newcommand{\setiso}{\mathscr{S}et_G^{Fin,Iso}}
\newcommand{\botimes}{\bigotimes}
\newcommand{\nmnl}{\ngh \M \bx \ngh \lu}
\newcommand{\nml}{\ngh(\M \bx \lu)}
\newcommand{\C}{\mathscr{C}}
\newcommand{\setisor}{\mathscr{S}et^{Fin,Iso}}
\newcommand{\uab}{\underline{ab}}

\newcommand{\co}{\colon}
\newcommand{\free}[1][G]{\textsc{free}^{#1}}

\newcommand{\tr}[1][G]{\textsc{tr}^{#1}}

\newcommand{\Map}{Map}

\begin{document}


\title{An Equivariant Tensor Product on Mackey Functors}

\author[Hill]{Michael~A.~Hill}
\email{mikehill@math.ucla.edu}

\author[Mazur]{Kristen~Mazur}
\email{kmazur@elon.edu}

\begin{abstract}
For all subgroups $H$ of a cyclic $p$-group $G$ we define norm functors that build  a $G$-Mackey functor from an $H$-Mackey functor.  We give an explicit construction of these  functors in terms of generators and relations based solely on the intrinsic, algebraic properties of Mackey functors and Tambara functors. We use these norm functors to define a monoidal structure on the category of Mackey functors where Tambara functors are the commutative ring objects. 
\end{abstract}



\maketitle

\section{Introduction}

For a finite group $G$, a commutative $G$-ring spectrum has norm maps that are multiplicative versions of the transfer maps. These maps are not seen in ordinary $G$-spectra. Moreover, we see the algebraic shadows of these norm maps in the zeroeth stable homotopy groups of commutative $G$-ring spectra: $\underline{\pi}_0$ of a $G$-spectrum is a Mackey functor (see for example \cite{alaska}), but if $X$ is a commutative $G$-ring spectrum, then $\underline{\pi}_0(X)$ is a Tambara functor \cite{Brun}. Thus, it is a Mackey functor with a ring structure (i.e$.$ a Green functor) and an extra class of norm maps that are the multiplicative analogues of the transfer maps.

In this paper we show that the relationship between Mackey functors and Tambara functors mirrors the relationship between $G$-spectra and commutative $G$-ring spectra.  We define an equivariant symmetric monoidal structure on the category of Mackey functors under which Tambara functors are the commutative monoid objects. The category of Mackey functors is symmetric monoidal, but the commutative ring objects under the symmetric monoidal product are not Tambara functors. Notably, they do not have norm maps.  

Hill and Hopkins have developed an appropriate notion of equivariant symmetric monoidal, calling it {\emph{$G$-symmetric monoidal}} \cite{HillHopkins}. They then call the commutative monoid objects under a $G$-symmetric monoidal structure the {\emph{$G$-commutative monoids}} \cite{HillHopkins}. We provide formal definitions of these concepts in Section 5. Hill  and Hopkins \cite{HillHopkins}, Ullman \cite{Ullman1}, and Hoyer \cite{Hoyer} have independently defined  $G$-symmetric monoidal structures on the category  $\mackg$ of $G$-Mackey functors. In this paper, for $G$ a cyclic $p$-group we define a very explicit $G$-symmetric monoidal structure on $\mackg$. The key to this structure is new norm functors $\ngh\co \mackh \to \mackg$ for all subgroups $H$ of $G$.
\begin{mainthm}\label{newmain}
Let $G$ be a cyclic $p$-group. For all subgroups $H$ of $G$ there is an explicit construction of a norm functor $\ngh\co \mackh \to \mackg$ that has the following properties.
\begin{enumerate}[a.]
\item  $\ngh$ is isomorphic to the composition of functors $N^G_KN^K_H$ whenever $H<K<G$.
\item $\ngh$ is strong symmetric monoidal.
\end{enumerate}
\end{mainthm}

 The word ``construction" is deliberate. Given an $H$-Mackey functor $\M$ we build a $G$-Mackey functor $\ngh\M$ based only on the intrinsic properties of Mackey functors and Tambara functors.  We then define the functor $\ngh$ via the map $\M \mapsto \ngh \M$ and use the collection of these norm functors $\{\ngh\text{ for all }H\le G\}$ to prove the following theorem.

\begin{thm}\label{mainthm2}Let $G$ be a cyclic $p$-group. There is a $G$-symmetric monoidal structure on $\mackg$ so that a $G$-Mackey functor is a $G$-commutative monoid if and only if it has the structure of a $G$-Tambara functor.\end{thm}

Because we restrict the Main Theorem and Theorem \ref{mainthm2} to cyclic $p$-groups we are able to explicitly describe the $G$-symmetric monoidal structure defined on $\mackg$. Thus, we can work with this construction to build intuition into how the structure operates.

In this regard our definition of the norm functors in the Main Theorem is analogous to the explicit definition of the tensor product. Even though defining the tensor product via the universal property is more elegant and at times easier to work with, the explicit definition helps us understand how the tensor product combines two vector spaces. Similarly, the explicit construction of norm functors in this paper allows us to understand what the $G$-symmetric monoidal structure actually does to the objects of $\mackg$. Moreover, while we feel strongly that we can extend our results at least to all finite abelian groups, many aspects would become so much more complicated that the results would no longer develop the reader's intuition for Mackey and Tambara functors.

 
 Even before the emergence of Tambara functors in equivariant stable homotopy theory, there was interest in developing a structure on $\mackg$ that supported Tambara functors as ring objects. During an open problem session at the 1996 Seattle Conference on Cohomology, Representations and Actions of Finite Groups T. Yoshida posed the problem of defining a tensor induction for Mackey functors that preserves tensor products of Mackey functors and satisfies Tambara's axioms for multiplicative transfer \cite{Benson}. The $G$-symmetric monoidal structure that we create is such a tensor induction.

 We organize this paper into the following sections. In Section 2 we provide two definitions of Tambara functors. The first is Tambara's original definition. The second is an axiomatic and constructive definition. We base the construction of the norm functors in the Main Theorem on the second definition. In Section 3 we build the norm functors, and in Section 4 we prove that these functors satisfy the Main Theorem. Finally, we prove Theorem \ref{mainthm2} in Section 5.


\section{Tambara Functors}
 
 Let  $\mathscr{S}et$ be the category of sets, and let $\set$ be the category of finite $G$-sets. Given morphisms $f\co X \to Y$ and $p\co A \to X$ in $\set$ we define the \textit{dependent product} $\prod_f A$ by $$\prod_fA = \left\{ (y, \sigma)\middle| \begin{array}{c} y \in Y, \\  \sigma\co f^{-1}(y) \to A \text{ is a map of sets }, \\ p\circ \sigma (x) = x \text{ for all } x \in f^{-1}(y)\end{array} \right\}.$$  The group $G$ acts on $\prod_f A$ by $\g(y,\sigma) = (\g y, \g\sigma)$ where $(\g\sigma)(x) = \g \sigma(\g^{-1}x)$. 

\begin{defn} \cite{Tambara} For all morphisms $f\co X \to Y$ and $p\co A \to X$ in $\set$ the {\emph{canonical exponential diagram}} generated by $f$ and $p$ is the commutative diagram below, where $h(y,\sigma) = y,$ the map $e$ is the evaluation map $e(x,(y,\sigma)) = \sigma(x)$ and $f'$ is the pullback of $f$ by $h$.\begin{displaymath}
\xymatrix{X \ar[d]_f & &A \ar[ll]_{p}& & X \times_Y \prod_f A   \ar[ll]_{e} \ar[d]^{f'} \\ Y &&& & \prod_fA \ar[llll]_{h}}
\end{displaymath}  We call a diagram in $\set$ that is isomorphic to a canonical exponential diagram an {\emph{exponential diagram}}. 

\end{defn}

\begin{defn}\label{origtfdef} \cite{Tambara} A \textit{$G$-Tambara functor} $\su$ is a triple $(S^*,S_*,S_{\star})$ consisting of two covariant functors $$ S_*\co \set \to \mathscr{S}et$$ $$S_{\star}\co \set \to \mathscr{S}et$$ and one contravariant functor $$S^*\co \set \to \mathscr{S}et$$ such that the following properties hold. \begin{enumerate}
\item All functors have the same object function $X \mapsto \su(X),$ and each $\su(X)$ is a commutative ring.
\item For all morphisms $f\colon X \to Y$ in $\set$ the map $S_*(f)$ is a homomorphism of additive monoids, $S_{\star}(f)$ is a homomorphism of multiplicative monoids and $S^*(f)$ is a ring homomorphism.
\item The pair $(S^*,S_*)$ is a $G$-Mackey functor and $(S^*,S_{\star})$ is a semi-$G$-Mackey functor.
\item (Distributive Law) If \begin{displaymath}
\xymatrix{X \ar[d]_f & &A \ar[ll]_{p}& & X'   \ar[ll]_{e} \ar[d]^{f'} \\ Y &&& &Y' \ar[llll]_{h}}
\end{displaymath} is an exponential diagram then the induced diagram below commutes.
$$\xymatrix{\su(X) \ar[d]_{f_{\star}} & &\su(A) \ar[ll]_{p_*} \ar[rr]^{e^*}& & \su(X')   \ar[d]^{f'_{\star}} \\ \su(Y) &&& & \su(Y') \ar[llll]_{h_*}}$$
\end{enumerate}
\end{defn}
 
 Given a morphism $f\co X\to Y$ in $\set$, we call $S_*(f)$ a \textit{transfer} map, $S^*(f)$ a \textit{restriction} map and $S_{\star}(f)$ a \textit{norm} map. In particular, if $f\co G/H \to G/K$ is a morphism between orbits, we denote the norm map $S_{\star}(f)$ by $N^K_H$, the transfer map $S_*(f)$ by $tr^K_H$, and the restriction map $S^*(f)$ by $res^K_H$. Further, given an element $g\in G$ there is also a morphism $c_g \co G/H \to G/gHg^{-1}$. Thus, we have an induced map $S_*(c_g)$ (which is equal to $S_{\star}(c_g)$ and $S^*(c_{g^{-1}})$) that we will denote by $c_g$.
 
The norm maps are not additive, but  there is an explicit formula for the norm of a sum in a Tambara functor. By Property 3 of Definition \ref{origtfdef} Tambara functors convert disjoint unions of $G$-sets to direct products of commutative rings. Since every finite $G$-set can be written as a disjoint union of orbits we only need  this formula for the norm maps $N^K_H$ for all subgroups $H<K \le G$. Moreover, since the norm map $N^K_H$ in a $G$-Tambara functor $\su$ must agree with $N^K_H$ in the $K$-Tambara functor that results from applying the forgetful functor $i_K^*\co \tambg \to \mathcal{T}amb_K$ to $\su$, it suffices to only state the formula for the norm $\ngh$ of a sum for any subgroup $H$ of $G$.

We will only use two cases of exponential diagrams, and for these it is helpful to use the following construction. Let $H$ be a subgroup of $G$ and let $X$ be a finite $H$-set. Then the $G$-set $Map_H(G,X)$ is the set of all $H$-equivariant maps $q \co G \to X$ with $G$-action $(g\cdot q)(g') = q(g'g)$ for all $g'$ and $g$ in $G$.

\begin{prop}\label{DePProdIso} Let $f \co G/H \to G/G$ be the crush map, and let $p \co A \to G/H$ be a morphism in $\set$. Then the dependent product $\prod_f A$ is canonically isomorphic to $Map_H(G,p^{-1}(eH))$.
\end{prop}

\begin{proof} The category of finite $G$-sets over $G/H$ is equivalent to $\mathscr{S}et^{Fin}_H$, and the pullback functor from the category of finite $G$-sets over $G/G$ to the category of finite $G$-sets over $G/H$ is canonically isomorphic to the restriction functor $\set \to \mathscr{S}et^{Fin}_H$. Moreover, the functor $(A \xrightarrow{p} G/H) \mapsto (\prod_f A \to G/G)$ is right adjoint to the pullback functor \cite{Tambara}. And, the functor $X \mapsto Map_H(G,X)$ is right adjoint to the restriction functor. It follows that these two right adjoints are canonically isomorphic. Therefore, for all $p \co A \to G/H$, the dependent product $\prod_f A$ is canonically isomorphic to $Map_H(G,p^{-1}(eH))$.
\end{proof}

\begin{thm}[Tambara Reciprocity for Sums]\label{TRforSumsFinCase} Let $G$ be a finite group with subgroup $H$, and let $\su$ be a $G$-Tambara functor.
\begin{enumerate}
\item If $*_1 \amalg *_2$ is the disjoint union of two single point $G$-sets then the following diagram of  rings  commutes. \small $$\xymatrix{\su(G/H) \ar[d]_{N^G_H}  &\su(G/H) \oplus \su(G/H) \ar[l]_-{\triangledown_*} \ar[r]^-{e^*} & \su(G/H \times Map_H(G,*_1\amalg*_2))   \ar[d]^{\pi_{\star}} \\ \su(G/G) & & \su(Map_H(G,*_1\amalg*_2)) \ar[ll]_-{h_*}}$$  \normalsize

\item For all $a$ and $b$ in $\su(G/H)$ $$\ngh(a+b) =\ngh(a) + \ngh(b) + T(-)$$ where $T(-)$ is a polynomial in transfers from proper subgroups, norms, and restrictions that is  universally determined by $G$ in the sense that it depends only on $G$ and not on $\su$, $a$, or $b$. Moreover, $T(-)$ lands in the ideal generated by the transfers.

\end{enumerate}
\end{thm}

\begin{proof} To prove part 1, given the fold map $\triangledown \co G/H \amalg G/H \to G/H$ and the crush map $f \co G/H \to G/G$, by Proposition \ref{DePProdIso}, the dependent product $\prod_f(G/H \amalg G/H)$  is canonically isomorphic to 
$$
Map_H(G,\ast_1\amalg\ast_2)\cong Map(G/H,*_1\amalg *_2).
$$
Therefore, the diagram below is an exponential diagram, and so the diagram in part 1 of the above theorem commutes by the Distributive Law of Definition \ref{origtfdef}.
$$\xymatrix{G/H \ar[d]_f & &G/H \amalg G/H\ar[ll]_-{\triangledown}& & G/H \times Map(G/H, {*_1}\amalg{*_2})   \ar[ll]_-{e} \ar[d]^{\pi} \\ G/G &&& & Map(G/H,{*_1} \amalg {*_2}) \ar[llll]_{h}}$$

Next, given $a$ and $b$ in $\su(G/H)$, by part 1, $\ngh \triangledown_*(a,b) = h_* \pi_{\star}e^*(a,b)$. Since $\ngh \triangledown_*(a,b) = \ngh(a+b)$, we can develop the formula given in part 2 by determining $h_*\pi_{\star}e^*(a,b)$.

We begin by examining the decomposition of $Map(G/H,*_1\amalg*_2)$ and $G/H \times Map(G/H,*_1\amalg*_2)$ into disjoint unions of $G$-orbits. We can rewrite this as $$
Map(G/H,*_1\amalg*_2)\cong G/G \amalg G/G \amalg \coprod_i Z_i
$$ 
where each $Z_i$ is of the form $G/H_i$ for some proper subgroup $H_i$. This union is difficult to determine, however it  suffices that it is universally determined by the group $G$, and none of the orbits $Z_i$ are isomorphic to $G/G$. Further, $G/H \times Map_H(G,*_1\amalg*_2)$ is isomorphic to $G/H \amalg G/H \amalg \coprod_i (G/H \times Z_i)$, and so the composition $h_* \pi_{\star} e^*$ is given below. $$\xymatrix{\su(G/H) \oplus \su(G/H) \ar[d]^-{e^*} \\ \displaystyle \su(G/H)\oplus \su(G/H) \oplus \bigoplus_i \su(G/H \times Z_i) \ar[d]^-{\pi_{\star}} \\ \displaystyle \su(G/G) \oplus \su(G/G) \oplus \bigoplus_i \su( Z_i) \ar[d]^{h_*} \\ \su(G/G)} $$
The map $e^*$ sends $(a,b)$ to $(a,b,\bm{x})$ where $\bm{x}$ is an element in $\bigoplus_i \su(G/H \times Z_i)$ that is universally determined by $G$.

Next, we determine $\pi_{\star}$. First, $\su(G/H \times Z_i)$ is isomorphic to $\coprod G/J_i$ where $G/J_i$ is the largest orbit that contains both $G/H$ and $Z_i$. Thus, let $f_i \colon \su(G/H \times Z_i) \to \su(Z_i)$ be the multiplication map composed with the appropriate norm map, and let $F \co \bigoplus_i \su(G/H \times Z_i) \to \bigoplus_i \su(Z_i)$ be $\bigoplus_i f_i$. Then $\pi_{\star} = \ngh \oplus \ngh \oplus F$, and hence $\pi_{\star}e^*(a,b) = \ngh(a) \oplus \ngh(b) \oplus F(\bm{x})$. 

Finally, $h_*$ is the identity map on the first two summands and the appropriate transfer maps on the remaining summands. Therefore, $h_*\pi_{\star}e^*(a,b)$ is $$\ngh(a) + \ngh(b) + T(-)$$ where $T(-)$ is as given in the above theorem.
\end{proof}

We explicitly describe the Tambara reciprocity relations for sums in a $C_4$-Tambara functor in Example~\ref{ex:TRFormulas}. In fact, when $G$ is a finite \textit{abelian} group we can determine more details regarding the polynomial $T(-)$. Thus, if $G$ is finite abelian, we can state an even more explicit formula for the norm of a sum in a $G$-Tambara functor.

\begin{thm}[Tambara Reciprocity for Sums When $G$ is Finite Abelian]\label{sumdef} Let $G$ be a finite abelian group and let $\su$ be a $G$-Tambara functor. If $H< G$, then for all $a$ and $b$ in $\su(G/H)$
\begin{eqnarray*}
\lefteqn{N^G_H(a+b) =} \\ && N^G_H(a) + N^G_H(b)  + \sum_{H<K<G} tr^G_{K}\left(\sum_{k=1}^{i_{K}}N^{K}_H\left((\uab)^{K}_k\right)\right) +  tr^G_H(g_H(a,b)) \nonumber
\end{eqnarray*}  where $g_H(a,b)$ is a  polynomial in some of the $W_G(H)$-conjugates of $a$ and $b$, and  each $(\uab)^{K}_k$ is a  monomial in some of the $W_G(K)$-conjugates of $a$ and $b$. These polynomials are universally determined by the group $G$ in the sense that they depend only on $G$ and not on $\su$. Each integer $i_{K}$ is also universally determined by $G$.\end{thm}

\begin{proof}
Given $a$ and $b$ in $\su(G/H)$, by Theorem \ref{TRforSumsFinCase}, $\ngh(a+b) = \ngh(a) + \ngh(b) + T(-)$, but since $G$ is finite abelian we can further determine $T(-)$ by developing a more concrete formula for the composition $h_*\pi_{\star}e^*(a,b)$. 

The entire analysis hinges on the structure of the dependent product, which as we showed in Proposition~\ref{DePProdIso}, is the $G$-set $Map(G/H,\ast_1\amalg\ast_2)$. In order to make later formulae more transparent, we write the set $\ast_1\amalg\ast_2$ as $\{a,b\}$ in all that follows. 

Because $G$ is an abelian group, the stabilizer of any element $f\in Map(G/H,\{a,b\})$ contains $H$. If $H< K$, then composition with the quotient map 
$$G/H\to G/K$$ 
induces an isomorphism
$$
Map(G/H,\{a,b\})^K\cong Map(G/K,\{a,b\}).
$$
Thus, since $i_K$ is the number of orbits of functions with stabilizer exactly $K$, the above isomorphism allows us to identify $i_K$ as the following cardinality:
$$
i_K=\left|Map(G/K,\{a,b\})-\bigcup_{K\subsetneq K'} Map(G/K',\{a,b\})\right|/[G:K].
$$
 This gives us our orbit decomposition
$$
Map(G/H,\{a,b\})\cong G/G\amalg G/G\amalg\coprod_{H<K<G} \coprod_{i=1}^{i_K}G/K.
$$



 
  Therefore, the composition $h_*\pi_{\star}e^*$ is given below where $S = \su(G/H)$.
  $$\xymatrix{S \oplus S \ar[d]^-{e^*} \\  
     \displaystyle S \oplus S \oplus  \bigoplus_{H<K<G}\bigg[\bigoplus_{k=1}^{i_{K}}\Big(\bigoplus_{|G/K|}S \Big)\bigg] \oplus \bigoplus_{j=1}^{i_H} \Big(\bigoplus_{|G/H|} S\Big)  \ar[d]^-{\pi_{\star}} \\ 
  \su(G/G) \oplus \su(G/G) \oplus \bigoplus_{H<K<G} \left( \bigoplus_{k=1}^{i_{K}} \su(G/K)\right) \ar[d]^{h_*} \oplus \bigoplus_{j=1}^{i_H} S  \\ \su(G/G)} $$
  
 The map $e^*$ sends $(a,b)$ to a sequence of $a$'s and $b$'s, the order of which is determined as follows.  The codomain of $e^*$ is grouped by the orbits of $Map_H(G,\{a,b\})$, and each grouping is indexed by the elements of $G/G$, $G/H$, or $G/K$ for $H<K<G$. For example, each $\bigoplus_{|G/H|} S$ corresponds to an orbit isomorphic to $G/H$, and thus we can write $\bigoplus_{|G/H|} S$ as $$S_e \oplus S_{g_1} \oplus \cdots \oplus S_{g_{|G/H|-1}}$$ where the indexing elements $e$, $g_1$, $\dots $, $g_{|G/H|-1}$ are representatives for the elements of $G/H$. To define $e^*$, we choose a representative map $f \colon G \to \{a,b\}$ from each orbit. (The choice of representative does not matter.) We then determine the image under $f$ of each of the indexing elements. For every indexing element $g_i$, if $f(g_i)=a$, then $e^*(a,b) =a$ on the corresponding summand, and if $f(g_i)=b$, then $e^*(a,b) =b$ on the summand. For example, the first summand corresponds to the constant function that sends all elements of $G$ to $a$. Hence, $e^*(a,b)=a$ on that summand. Similarly,  $e^*(a,b)=b$ on the second summand.

Next, $\pi_{\star}$ is a combination of norm maps and maps that send a direct sum to a product over a Weyl action. More specifically, let the map $\pi_H \co \bigoplus_{|G/H|} \su(G/H) \to \su(G/H)$ be $$(s_e, s_{g_1},s_{g_2},\dots s_{g_{|G/H|-1}}) \mapsto s_eg_1 s_{g_1} g_2s_{g_2} \cdots g_{|G/H|-1}s_{g_{|G/H|-1}}.$$ The map $\pi_{K}\co \bigoplus_{|G/K|} \su(G/H) \to \su(G/H)$ is analogous. Then $\pi_{\star}$ is $$\ngh \oplus \ngh \oplus \bigoplus_{H<K<G} \bigoplus_{k=1}^{i_{K}} \left[N^{K}_H \circ \pi_{K}\right]_k \oplus \bigoplus_{j=1}^{i_H} (\pi_H)_j ,$$ where $\ngh$ and $N^{K}_H$ are the norm maps.

 Thus, if each $(\uab)^{K}_j$ is as defined in the theorem and $(\uab)^H_j$ is defined analogously, then $\pi_{\star}e^*(a,b)$ equals $$\ngh(a) \oplus \ngh(b) \oplus \bigoplus_{H<K<G} \Big(\bigoplus_{k=1}^{i_{K}} N_H^{K}((\uab)^{K}_k)\Big) \oplus \bigoplus_{j=1}^{i_H} (\uab)^H_j .$$ Finally, $h_*\pi_{\star}e^*(a,b)$ is $$N^G_H(a) + N^G_H(b)  +   \sum_{H<K<G} tr^G_{K}\left(\sum_{k=1}^{i_{K}}N^{K}_H\left((\uab)^{K}_k\right)\right) +tr^G_H(g_H(a,b)) $$ where  $g_H(a,b) = \sum_{j=1}^{i_H} (\uab)^H_j$ and thus is as defined in the above theorem.
\end{proof}

We use the following corollary frequently in the upcoming sections.

\begin{cor}[Tambara Reciprocity for Sums When $G$ is a Cyclic $p$-Group]\label{cor:TRsums} Let $G=C_{p^n}$ be a cyclic $p$-group, let $H=C_{p^k}$ be a subgroup of $G$, and let $\su$ be a $G$-Tambara functor. Then for all $a$ and $b$ in $\su(G/H)$, \begin{multline*}
\lefteqn{N^G_H(a+b) =} \\  N^G_H(a) + N^G_H(b)  + \sum_{H<C_{p^j}<G} tr^G_{C_{p^j}}\left(\sum_{f\in\mathcal I_{j}/G}N^{C_{p^j}}_H\left((\uab)_f\right)\right) +  tr^G_H(g_H(a,b)),
\end{multline*}
where 
$$
\mathcal I_j=\big(\Map(G/C_{p^j},\{a,b\})-\Map(G/C_{p^{j+1}},\{a,b\})\big),
$$
where
$$
(\uab)_f=\prod_{gC_{p^j}\in G/C_{p^j}} gf(gC_{p^j}),
$$
and where
$$
g_H(a,b)=\sum_{f\in\mathcal I_k/G} (\uab)_f = \sum_{f\in\mathcal I_k/G}\left(\prod_{gH \in G/H}g f(gH)\right).
$$
\end{cor}

We also have a formula for the norm of a transfer in a Tambara functor. As we did for the norm of a sum, we first provide a general formula that holds for all finite groups. We then give a more detailed formula for the norm of a transfer in a $G$-Tambara functor when $G$ is finite abelian.

\begin{thm}[Tambara Reciprocity for Transfers]\label{TRfortransfersFinCase} Let $G$ be a finite group with subgroups $H' < H$, and let $\su$ be a $G$-Tambara functor. 
\begin{enumerate}
\item  The following diagram of rings commutes.
$$\xymatrix{\su(G/H) \ar[d]_{\ngh} & &\su(G/H') \ar[ll]_{tr^H_{H'}}\ar[rr]^-{e^*}& & \su(G/H \times Map_H(G,H/H'))    \ar[d]^{\pi_{\star}} \\ \su(G/G) &&& &\su(Map_H(G,H/H')) \ar[llll]_-{h_*}}$$

\item For all $x$ in $\su(G/H')$, $$\ngh tr^H_{H'}(x) = T(-)$$ where $T(-)$ is an element in the subgroup of $\su(G/G)$ generated by all transfer terms. Moreover, it is universally determined by $G$.
\end{enumerate}
\end{thm}

\begin{proof} By Proposition \ref{DePProdIso}, the dependent product of the composition $G/H' \xrightarrow{p} G/H \xrightarrow{f} G/G$ is canonically isomorphic to $Map_H(G,H/H')$. It follows that the diagram below is an exponential diagram.
$$\xymatrix{G/H \ar[d]_f & &G/H' \ar[ll]_p& & G/H \times Map_H(G,H/H')   \ar[ll]_-{e} \ar[d]^{\pi} \\ G/G &&& &Map_H(G,H/H') \ar[llll]_-{h}}$$

Next we prove part 2 of the theorem. The $G$-set $Map_H(G,H/H')$ has no $G$-fixed points. Thus, there are no copies of $G/G$ in its decomposition into $G$-orbits. This means that in the commutative diagram of rings given in part 1, the map $h_*$ is a sum of various transfer maps. Hence, for all $x$ in $\su(G/H')$, $h_*\pi_{\star}e^*(x)$ is some element in the subgroup of $\su(G/G)$ generated by the transfer terms.
\end{proof}

\begin{thm}[Tambara Reciprocity for Transfers When $G$ is Finite Abelian]\label{normoftrdef} Let $G$ be a finite abelian group, and let $\su$ be a $G$-Tambara functor. For all subgroups $H'<H<G$ and all $x$ in $\su(G/H')$
\begin{eqnarray*}
N^G_Htr^H_{H'}(x) &=& tr^G_{H'}(f(x))+  \sum_{K \in \Omega} \left(\sum_{j=1}^{q_K}tr^G_K N^K_{H'}(\underline{x}_j)\right)
\end{eqnarray*}
where $\Omega$ is the set of all subgroups $K$ of $G$ such that $H'=H \cap K$, $f(x)$ is a polynomial in  some of the  $W_G(H')$-conjugates of $x$, and each $\underline{x}_j$ is a monomial in some of the $W_G(H')$-conjugates of $x$. Moreover, $f(x)$, $\underline{x}_j$, and the integers $q_K$ are universally determined by the group $G$.\end{thm}

\begin{proof}Since $G$ is abelian we can create a more complete picture of the composition $h_*\pi_{\star}e^*(x)$ given  in Theorem \ref{TRfortransfersFinCase}. In particular, we are now able to better describe the map $\pi_{\star}$.

First, as a $G$-set, $Map_H(G,H/H')$ is isomorphic to $$\coprod_{i=1}^r (G/H')_i \amalg \coprod_{K \in \Omega} \Big( \coprod_{j=1}^{q_K} (G/K)_j\Big).$$ We do not need to know the exact values for $r$ and each $q_K$, since it suffices that they are universally determined by $G$. So, $G/H \times Map_H(G,H/H')$ is isomorphic to $$\coprod_{i=1}^r \Big( \coprod_{|G/H|} G/H'\Big)_i \amalg \coprod_{K \in \Omega} \bigg[ \coprod_{j=1}^{q_K} \Big(\coprod_m G/H'\Big)_{j}\bigg]$$ where $m = \frac{|G/K||G/H|}{|G/H'|}$.  

Thus, $h_*\pi_{\star}e^*$ becomes the composition below where $S = \su(G/H')$. 
$$\xymatrix{S  \ar[d]^-{e^*} \\  
     \bigoplus_{i=1}^r\left(\bigoplus_{|G/H|}S\right) \oplus \bigoplus_{K \in \Omega} \left[\bigoplus_{j=1}^{q_K} \left(\bigoplus_m S\right)\right]  \ar[d]^-{\pi_{\star}} \\ 
  \bigoplus_{i=1}^r  S  \oplus \bigoplus_{K \in \Omega} \left[\bigoplus_{j=1}^{q_K} \su(G/K)\right] \ar[d]^{h_*} \\ \su(G/G)} $$

To describe the map $\pi_{\star}$ let  $$\pi_i\co \bigoplus_{|G/H|} S\to S$$ be the map that sends $(s_e, s_{g_1},\dots , s_{g_{|G/H|-1}})$ to a specific product $$g_{t_e}s_e g_{t_1}s_{g} \cdots g_{t_{|G/H|-1}}s_{g^{|G/H|-1}}$$ where $g_{t_{i}}s_{g_i}$ is a $W_G(H')$-conjugate  of $s_{g_i}$. While it is difficult to give further detail on this product of Weyl conjugates, it suffices that it is universally determined by the group $G$. Define $\pi_j\co \bigoplus_m S \to S$ analogously. Then $ \pi_{\star} = \bigoplus_{i=1}^r \pi_{i} \oplus \bigoplus_{j=1}^{q_K} (N^K_{H'}\circ \pi_j)$. 

The map $e^*$ is the diagonal map and $h_*$ is a composition of addition and transfer maps. Therefore, $$N^G_Htr^H_{H'}(x) = h_*\pi_{\star}e^*(x) = tr^G_{H'}(f(x)) + \sum_{K \in \Omega} \left(\sum_{j=1}^{q_K}tr^G_K N^K_{H'}(\underline{x}_j)\right)$$ where $f(x)$ is a polynomial in some of the $W_G(H')$-conjugates of $x$, and each $\underline{x}_j$ is a monomial in some of the $W_G(H')$-conjugates of $x$. 
\end{proof}

When $G$ is a cyclic $p$-group, since all subgroups are nested, $H' = H \cap K$ if and only if $K=H'$. Thus, in this case the Tambara reciprocity formula for transfers simplifies nicely.

\begin{cor}[Tambara Reciprocity for Transfers When $G$ is a Cyclic $p$-Group]\label{cor:TRtransfers} Let $G=C_{p^n}$ be a cyclic $p$-group with chosen generator $\g$, and let $H'<H$ be subgroups of $G$. Then for all $x$ in $\su(G/H')$,
\begin{eqnarray*}
N^G_Htr^H_{H'}(x) &=& tr^G_{H'}(f(x))
\end{eqnarray*}
where $f(x)$ is a polynomial in some of the Weyl conjugates of $x$. More specifically,  $$f(x) = \sum_{s=1}^r \prod_{i=0}^{|G/H|-1} \g^i \g^{m_{i,s}} x$$ where each $\g^{m_{i,s}}x$ is a $W_H(H')$-conjugate of $x$ and the integers $r$ and $m_{i,s}$ are universally determined by $G$. 
\end{cor} 

 
 \begin{fact}\label{monfacts}
The following fact regarding Tambara reciprocity  will allow us to create Tambara reciprocity-like relations in the Mackey functor $\ngh\M$ that we define in the next section. Since we only use this fact when $G$ is a cyclic $p$-group we restrict our attention to this case. It appears however that this fact holds for all finite abelian groups.

When $G$ is a cyclic $p$-group the monomials of $g_H(a,b)$ and $f(x)$ and the monomials $(\uab)_f$ given in Corollaries \ref{cor:TRsums} and \ref{cor:TRtransfers} do not contain repeated factors. If $\g$ generates $G$, then for every $\g^m$ the elements  $\g^m a$ and $\g^m b$ appear at most once in any monomial in the formula for the norm of a sum, and  it is impossible for both $\g^m a$ and $\g^m b$ to appear in the same monomial. Similarly, the element $\g^mx$ appears at most once in any monomial in the formula for the norm of a transfer. These facts follow from the formulas for the  maps $\pi_{\star}$  in the proofs of Theorems \ref{sumdef} and \ref{normoftrdef}.\end{fact}

Let $G$ be a finite group and let $\og$ be the orbit category of $G$. Motivated by an analogous definition of Mackey functors from \cite{webb}, we also define $G$-Tambara functors as follows.

\begin{defn}\label{axiomatictfdef}
Let $G$ be a finite group. A $G$-Tambara functor $\su$ consists of a collection of commutative rings $$\{\su(G/H): G/H \in \og\}$$   along with the following maps for all orbits $G/H$ and $G/K$ such that $H<K$ and all $g$ in $G$: 
\begin{itemize}\vspace{-5pt}
\item the {\emph{restriction}} map  $res^K_H\co  \su(G/K)\to \su(G/H)$, \vspace{-5pt}
\item the {\emph{transfer}} map $tr^K_H\co \su(G/H) \to \su(G/K)$,  \vspace{-5pt}
\item the {\emph{norm}} map $N^K_H\co \su(G/H) \to \su(G/K)$, and \vspace{-5pt}
\item the \textit{conjugation} map $c_g \co \su(G/H) \to \su(G/gHg^{-1})$.
\end{itemize}
These rings and maps satisfy the following conditions.
\begin{enumerate}
\item All restriction maps are ring homomorphisms, all transfer maps are homomorphisms of additive monoids, and all norm maps are homomorphisms of multiplicative monoids.
\item \label{tftrans} (Transitivity) For all $H'<H<K$ and all $g$ and $h$ in $G$ \vspace{-3pt}
\begin{eqnarray*}
res^K_{H'} &=& res^H_{H'}res^K_H\\ 
tr^K_{H'} &=& tr^K_Htr^H_{H'} \\
N^K_{H'} &=& N^K_HN^H_{H'} \\
c_gc_h &=& c_{gh}.
\end{eqnarray*}
\item (Frobenius Reciprocity) If $H<K$, then for all $x$ in $\su(G/H)$ and $y$ in $\su(G/K)$  $$ytr^K_H(x) = tr^K_H(res^K_H(y)x).$$
\item If $H<K$ and $g \in G$, then 
\begin{eqnarray*}
c_gres^K_H &=& res^{gKg^{-1}}_{gHg^{-1}}c_g \\
c_gtr^{K}_{H} &=& tr^{gKg^{-1}}_{gHg^{-1}}c_g \\
c_gN^{K}_{H} &=& N^{gKg^{-1}}_{gHg^{-1}}c_g.
\end{eqnarray*}
\item For subgroups $H'$ and $H$ of $K$, let $[H\backslash K/H']$ denote a set of representatives in $G$ for the double cosets $H\backslash K/H'$. Then
\begin{eqnarray*}
res^K_Htr^K_{H'}&=& \sum_{g \in [H\backslash K/H']} tr^H_{H \cap gH'g^{-1}} c_g res^{H'}_{gHg^{-1}\cap H'} \\
res^K_HN^K_{H'} &=& \prod_{g \in [H\backslash K/H']} N^H_{H \cap gH'g^{-1}} c_g res^{H'}_{gHg^{-1}\cap H'}.
\end{eqnarray*}
\item  (Tambara Reciprocity) $\su$ satisfies Theorems \ref{TRforSumsFinCase} and \ref{TRfortransfersFinCase}.
\end{enumerate}\end{defn}

If $G$ is a finite group, and  $\su$ is a $G$-Tambara functor as defined in Definition \ref{origtfdef}, showing that $\su$ satisfies Definition \ref{axiomatictfdef} is straightforward.

\begin{rmk}[\textbf{Weyl Actions}]\label{resN} Every $\su(G/H)$ is equipped with actions of the Weyl groups $W_K(H)$ for all subgroups $K$ such that $H<K\le G$. Here we state a few properties of these actions that we subsequently use to define the Mackey functor $\ngh\M$. By Property 4 of Definition \ref{axiomatictfdef}, for all $g$ in $W_K(H)$, all $x$ in  $\su(G/H)$, and all $y$ in $\su(G/K)$, $g res^K_H(x) = res^K_H(x)$, $tr^K_H(gx) = tr^K_H(x)$ and $N^K_H(gx) = N^K_H(x)$. Moreover, when $G$ is abelian, by Property 5, $$res^K_Htr^K_H(x) = \sum_{g \in W_K(H)} g x \hspace{1cm} \text{ and } \hspace{1cm} res^K_HN^K_H(x) = \prod_{g \in W_K(H)} g x.$$

\end{rmk}

We collect the properties of Definition \ref{axiomatictfdef} into a lattice-like diagram. For example, a $C_2$-Tambara functor is pictured in Figure \ref{tfdiag}. We do not draw the Weyl action.
\begin{figure}[ht]
\begin{displaymath}
\xymatrix@1{ \su(C_{2}/C_{2})  \ar@/_2pc/[d]_{res^{C_2}_e}   \\ \text{ } \su(C_{2}/e)\ar@/_2pc/[u]_{tr^{C_2}_e} \ar[u]^{N^{C_2}_e} &\text{ } }
\end{displaymath}
\caption{$\su$ is a $C_2$-Tambara Functor}
\label{tfdiag}
\end{figure}


\begin{ex}\label{ex:TRFormulas} We can find explicit  Tambara reciprocity formulas. For example,
let $\su$ be a $C_4$-Tambara functor, let $a$ and $b$ be elements in $\su(C_4/e)$, and let $\g$ generate $C_4$. Then  $$N^{C_4}_e(a+b)=  N^{C_4}_e(a) + N^{C_4}_e(b) + tr^{C_4}_{C_2}\left(N^{C_2}_e((\underline{ab})^{C_2}_1)\right) + tr^{C_4}_e(g_e(a,b))$$
where $(\underline{ab})^{C_2}_1 = a \g b$ and $$ g_e(a,b) = a\g a\g^2  a\g^3  b +  b \g  b\g^2  b\g^3  a +a\g  b\g^2  b\g^3  a.$$   
Further, for $x$ in $\su(C_4/e)$, $N^{C_4}_{C_2}tr^{C_2}_e(x) = tr^{C_4}_e(f(x))$ where $f(x) = x\g x$.
\end{ex}


\section{Constructing the Norm Functors}

From here on, let $G$ be a cyclic group of order \(p^n\) with chosen generator $\g$. In this section,  for a subgroup $H=C_{p^k}$ of $G$ and $H$-Mackey functor $\M$, we build a $G$-Mackey functor $\ngh\M$ that we will use to define the norm functor $\ngh\co \mackh \to \mackg$. Our Mackey functor norm will be built inductively, using the linear order of the subgroups of a cyclic \(p\)-group. This allows for a clean, conceptual description of the norm.


Our definition of $\ngh\M$ is motivated by the explicit description of the symmetric monoidal product in the category of $G$-Mackey functors. This product is called the {\emph{box product}} $\bx$. The category theoretic definition of the box product can be found in \cite{LewisGF} or \cite{Ventura}, and while this definition is central to the theory of Mackey functors, it is difficult to actually compute the box product of two Mackey functors. 

Lewis gives an explicit construction of the box product for $C_p$-Mackey functors, and this description immediately generalizes to other cyclic $p$-groups. Despite its constraint to $C_p$-Mackey functors, this definition allows us to visualize the box product of two Mackey functors. In a similar vain, even though we only define $\ngh\M$ when $G$ is a cyclic $p$-group, our construction provides insight into the norm functors that is otherwise difficult to see.

\begin{defn}\label{bxproddef}  [\cite{Shulman}, \cite{LewisLT}] Given $C_p$-Mackey functors $\M$ and $\lu$ we define their box product $\M \bx \lu$ by the diagram in Figure \ref{bxproddiag}.

\begin{figure}[h]\label{bxproddiag} 
\begin{displaymath}
\xymatrix@1{ \big[\M(C_p/C_p) \otimes \lu(C_p/C_p) \oplus \overbrace{(\M(C_p/e) \otimes \lu(C_p/e))/_{W_{C_p}(e)}}^{Im(tr^{C_p}_e)}\big]/_{FR} \ar@/_2pc/[d]_{res^{C_p}_e}   \\ \text{ } \M(C_p/e) \otimes \lu(C_p/e) \ar@/_2pc/[u]_{tr^{C_p}_e}  &\text{ } }
\end{displaymath}
\caption{The Box Product of $C_p$-Mackey Functors}
\label{bxproddiag}
\end{figure}

The transfer map is the quotient map onto the second summand. The restriction map is induced from $res^{C_p}_e \otimes res^{C_p}_e$ on the first summand and by the trace of the Weyl action on the second summand. The Weyl action on $\M(C_p/e) \otimes \lu(C_p/e)$ is the diagonal action. Finally, $FR$ is the {\emph{Frobenius reciprocity}} submodule and is generated by all elements of the form $$m' \otimes tr^{C_p}_e(l) - tr^{C_p}_e(res^{C_p}_e(m') \otimes l)$$ and $$tr^{C_p}_e(m) \otimes l' - tr^{C_p}_e(m \otimes res^{C_p}_e(l'))$$ for all $m'$  in $\M(C_p/C_p)$, $l'$ in $\lu(C_p/C_p)$, $m$ in $\M(C_p/e)$, and $l$ in $\lu(C_p/e)$. \end{defn}

If $G$ is a cyclic $p$-group $C_{p^n}$ we can extend the above definition to the box product of $G$-Mackey functors as follows. Because all subgroups of $G$ are nested, for any $G$-Mackey functor and any subgroups $H<C_{p^j} \le G$, $tr^{C_{p^j}}_H = tr^{C_{p^j}}_{C_{p^{j-1}}}tr^{C_{p^{j-1}}}_H$. Therefore, if $\M$ and $\lu$ are $G$-Mackey functors, we can inductively build $\M \bx \lu$ so that 
\begin{eqnarray*}
\lefteqn{(\M \bx \lu)(G/C_{p^j}) =} \\ && \big[\M(G/C_{p^j}) \otimes \lu(G/C_{p^j}) \oplus \underbrace{(\M \bx \lu)(G/C_{p^{j-1}})/_{W_{C_{p^j}}(C_{p^{j-1}})}}_{Im(tr_{j-1}^{j})}  \big]/_{FR}.
\nonumber
\end{eqnarray*} The restriction and transfer maps are identical to those given in Definition \ref{bxproddef}.

The above definition naturally extends to an $m$-fold box product for any positive integer $m$. It also gives us another way to view the universal property of the box product. This is equivalent to the usual description of the universal property; we find this a helpful reformulation. The proof is immediate from Lewis' description of the box product.

\begin{prop}\label{boxprodmaps}\cite{LewisGF,Shulman}
Let $\M$, $\lu$, and $\underline{N}$ be $G$-Mackey functors. Maps $\M \bx \lu \to \underline{N}$ are in natural bijective correspondence with the following data:
A collection of Weyl equivariant maps
\[
f_j\colon \M(G/C_{p^j})\otimes \lu(G/C_{p^j})\rightarrow \underline{N}(G/C_{p^j})
\]
for all $0\le j \le n$ that satisfy the following compatibility conditions.
\begin{enumerate}
\item For all $j$, we have
\[
res_{j-1}^{j}\circ f_j=f_{j-1}\circ (res_{j-1}^{j} \otimes res_{j-1}^j),
\]
\item for all $j$, we have
\[
f_j\circ (tr_{j-1}^{j}\otimes Id)=tr_{j-1}^{j}\circ f_{j-1}\circ (Id\otimes res_{j-1}^{j}),
\]
\item and for all $j$, we have 
\[
f_j\circ (Id\otimes tr_{j-1}^{j})=tr_{j-1}^{j}\circ f_{j-1}\circ (res_{j-1}^j\otimes Id).
\]
\end{enumerate}
\end{prop}

Alternatively, we can define $\M \bx \lu$ as a quotient of Mackey functors. If we let $\textsc{tens}\M \lu$ be the Mackey functor with 
$$
(\textsc{tens} \M\lu)(G/K) = \left[\M(G/K) \otimes \lu(G/K)\right] \oplus Im(tr)
$$ for all subgroups $K$ of $G$,
then the Frobenius reciprocity submodules form a subMackey functor $\textsc{fr}\M\lu$ of $\textsc{tens}\M\lu$. We can then define $\M \bx \lu$ by $$\M \bx \lu := \textsc{tens}\M\lu/\textsc{fr}\M\lu.$$

We play a similar game to define $\ngh\M$. We start by building a ``free" $G$-Mackey functor from an $H$-Mackey functor $\M$. Since we want $\ngh\M$ to look like a Tambara functor we will quotient the ``free" Mackey functor by a subMackey functor (analogous to $\textsc{fr}\M\lu$) that creates Tambara reciprocity-like relations. 

\subsection{The ``free'' $G$-Mackey functor}

\begin{defn}[The ``free" $G$-Mackey functor] Let $H=C_{p^k}$ be a subgroup of $G$ and let $\M$ be an $H$-Mackey functor.  Define the $G$-Mackey functor $\free \M$ as follows.

  For all subgroups $H'$  of $H$  define $$(\free \M)(G/H'):= \M^{\bx |G/H|}(H/H'),$$  and if $H''<H'\le H$ then  $res^{H'}_{H''}$ and $tr^{H'}_{H''}$ are the box product restriction and transfer maps given in Definition \ref{bxproddef}. The Weyl action is via tensor induction.
  
Now let $K=C_{p^j}$ be a subgroup such that $H<K\le G$. Define $(\free\M)(G/K)$ inductively by 
$$
(\free \M)(G/K) := \Z\{\Map\big(G/K,\M(H/H)\big)\} \oplus Im(tr^{j}_{j-1}),
$$ 
where 
$$
Im(tr^{j}_{j-1})=\left((\free \M)(G/C_{p^{j-1}})\right)/_{W_{K}(C_{p^{j-1}})}.
$$
The transfer map $tr^{j}_{j-1}$ is the canonical quotient map onto that summand. The Weyl action is via the action on $\Map(G/K,\M(H/H))$ by precomposition and the induced action on the other summand. 

The restriction on $Im(tr^{j}_{j-1})$ is defined by the universal formula
$$
res^j_{j-1}tr^j_{j-1}(x) = \sum_{g\in W_{C_{p^j}}(C_{p^{j-1}})} g x
$$
in a Mackey functor. On the ``free'' summand, when $j>k+1$ we take
$$
res^j_{j-1}\colon \Map\big(G/K,\M(H/H)\big)\to \Map\big(G/C_{p^{j-1}},\M(H/H)\big)
$$
to be the map induced by the quotient map
$$
G/C_{p^{j-1}}\to G/K.
$$
Finally, the restriction map $res^{k+1}_k$ on the free summand is given by the composite 
\begin{multline*}
\Z\{\Map\big(G/C_{p^{k+1}},\M(H/H)\big)\}\to 
\Z\{\Map\big(G/H,\M(H/H)\big)\}\to \\ 
\M(H/H)^{\otimes |G/H|}\to \M^{\Box |G/H|}(H/H),
\end{multline*}
where the latter maps are the canonical maps. 
\end{defn}

\begin{rmk}\label{DefofN} There is a map $N \co \M(H/H) \to (\free\M)(G/G)$ defined by the composition $$\M(H/H) \to \Z\{\M(H/H)\} \to (\free\M)(G/G)$$ that sends an element in $\M(H/H)$ to the corresponding generator in the free summand in $(\free\M)(G/G)$. This map $N$ will be the universal norm map, and in Section 5, we use it to define the internal norm maps of a Tambara functor. The restriction map on this summand (and the other factors) is defined to mirror the relation in a Tambara functor
$$
res^{j}_{j-1}N^{j}_{j-1}(x) = \prod_{g\in W_{C_{p^j}}(C_{p^{j-1}})} g x.
$$
Hence, with this idea in mind we will use the notation $N(-)$ to denote generators  of the free summand.
\end{rmk}

In order to make some computations more explicit, we need to introduce some notation.
Since whenever $H'<H$, $(\free\M)(G/H') = \M^{\bx |G/H|}(H/H')$, we index $(\free\M)(G/H')$ by elements of $G/H$.  Moreover, 
$$
\M^{\bx |G/H|}(H/H') = \left(  \M(H/H')^{\otimes |G/H|} \oplus Im(tr)\right)/_{FR},
$$ 
and we denote a simple tensor in the $ \M(H/H')^{\otimes |G/H|}$ summand of this module by 
$$
\bfm^{\otimes|G/H|} = m_e \otimes m_{\g} \otimes \cdots \otimes m_{\g^{|G/H|-1}}.
$$

Similarly, $\Map\big(G/K,\M(H/H)\big)$ in the free summand of $(\free\M)(G/K)$ is  the product over $G/K$ and can be indexed by elements of $G/K$. We denote an element in $\Map\big(G/K,\M(H/H)\big)$ by 
\begin{equation*}\label{notation} 
\bfm^{\times |G/K|} = m_e\times  m_{\g} \times \cdots \times m_{\g^{|G/K|-1}}.
\end{equation*}  
Let  $N(\bfm^{\times|G/K|})$ denote the corresponding generator of  the free summand $\Z\{\Map\big(G/K,\M(H/H)\big)\}$ in $(\free\M)(G/K)$.

This gives us a universal property for $\free\M$ that follows immediately from the universal property of the free abelian group functor. 

\begin{prop}
A map $\free\M\to\M'$ is given by two pieces of data:
\begin{enumerate}
\item A map of Mackey functors $\M^{\Box |G/H|}\to i_H^\ast \M'$, where the source is the analogue of tensor induction in Mackey functors, and for both sides, the Weyl action by $W_H(H')$ comes equipped with an extension to an action of $W_G(H')$,
\item for $K=C_{p^j}$ such that $H<K\le G$, $G/K$-equivariant set maps
$$
f_j\colon\Map(G/K,\M(H/H))\to \M'(G/K)
$$

that satisfy the following properties.
\begin{enumerate}
\item For each $k<j-1<n$, we have
$$
res_{j-1}^j\circ f_j=f_{j-1}\circ res_{j-1}^{j}
$$
\item when $j-1=k$ we have
$$
\xymatrix{
{\Map\big(G/C_{p^{k+1}},\M(H/H)\big)}
	\ar[r]^(.625){f_{k+1}}
    \ar[d]
    &
{\M'(G/C_{p^{k+1}})}
	\ar[d]
    \\
{\M^{\Box |G/H|}(H/H)}
	\ar[r]
    &
{\M'(H/H).}
}
$$
\end{enumerate}
\end{enumerate}
\end{prop}

\subsection{The Tambara Reciprocity SubMackey Functor $\tr\M$} \label{fun}

As discussed in Remark \ref{DefofN}, in Section 5 we will use $\ngh\M$ to define the internal norm maps of a Tambara functor, so $\ngh\M$ must reflect the Tambara reciprocity property of Definition \ref{axiomatictfdef}. Hence, we will define $\ngh\M$ by taking the quotient of $\free\M$ with a subMackey functor to produce relations that mimic the Tambara reciprocity relations of a Tambara functor. 
Thus, we need to quotient $\free \M$ by a subMackey functor in order to identify the analogues of ``the norm of a sum" and ``the norm of a transfer" in the free summand with appropriate elements in the transfer summand. We will call this subMackey functor the \textit{Tambara reciprocity subMackey functor} $\tr\M$.

We will define $\tr\M$ inductively. We begin with a lemma identifying $\free\M$.
\begin{lemma}\label{freegiso}
For any $K=C_{p^j}$ such that $H<K\le G$ we have a natural isomorphism
$$
i_{K}^\ast \free \M\cong \big(\free[{K}]\M\big)^{\Box |G/K|}.
$$
\end{lemma}

\begin{proof} We will use induction on the order of the group. First, by definition of $\free\M$, for $H' \le H$, $(i_{K}^*\free\M)(K/H')$ is isomorphic to $(\free[{K}]\M)^{\bx |G/K|}(K/H')$. Then by induction it remains to show that $(i_{K}^*\free\M)(K/K)$ is isomorphic to $(\free[{K}]\M)^{\bx |G/K|}(K/K)$. First, we have $$(i_{K}^*\free\M)(K/K) = \Z\{Map(G/K, \M(H/H))\} \oplus Im(tr^{j}_{j-1}),$$ and 
\begin{multline*}
(\free[{K}])^{\bx |G/K|}(K/K) = \\ \left(\left(\Z\{Map(K/K, \M(H/H))\} \oplus Im(tr^{j}_{j-1})\right)^{\otimes |G/K|} \oplus Im^{\bx}(tr^{j}_{j-1}) \right)/_{FR}.
\end{multline*}

Then by Frobenius reciprocity we identify all transfer terms of $$\left(\Z\{Map(K/K,\M(H/H))\} \oplus Im(tr^{j}_{j-1})\right)^{\otimes |G/K|}$$ with elements of $ Im^{\bx}(tr^{j}_{j-1})$. Thus, $(\free[{K}])^{\bx |G/K|}(K/K)$ is isomorphic to $$\Z\{Map(K/K,\M(H/H))\}^{\otimes |G/K|} \oplus Im^{\bx}(tr^{j}_{j-1}).$$ Then $\Z\{Map(K/K, \M(H/H))\}^{\otimes |G/K|}$ is isomorphic to $\Z\{Map(G/K,\M(H/H))\}$, and by induction  $Im^{\bx}(tr^{j}_{j-1})$ is isomorphic to the $Im(tr^{j}_{j-1})$ summand of $(i_{K}^*\free\M)(K/K)$. Therefore, $(\free[{K}])^{\bx |G/K|}(K/K)$ is isomorphic to \newline $(i_{K}^*\free\M)(K/K)$.
\end{proof}

Our definition of $\tr\M$ uses a similar inductive description: we build $\tr\M$ so that for $H<K\le G$ we have an isomorphism
$$
i_{K}^\ast\free\M/\tr\M\cong \big(\free[K]\M/\tr[K]\M\big)^{\Box |G/K|}.
$$

We start by defining the elements that generate the submodule $(\tr\M)(G/G)$ of $(\free\M)(G/G)$. These generators are of two types, one of which will create relations like Tambara reciprocity for sums in $\ngh\M$ and the other will create relations like Tambara reciprocity for transfers. 

To develop the first type of generators, recall by Corollary \ref{cor:TRsums} that when $G$ is a cyclic $p$-group, the Tambara reciprocity formula for sums is given by
\begin{multline*}
\lefteqn{N^G_H(a+b) =} \\  N^G_H(a) + N^G_H(b)  + \sum_{H<C_{p^j}<G} tr^G_{C_{p^j}}\left(\sum_{f\in\mathcal I_{j}/G}N^{C_{p^j}}_H\left((\uab)_f\right)\right) +  tr^G_H(g_H(a,b)).
\end{multline*}

Thus, we will develop elements in $(\free\M)(G/G)$ that are analogous to these formulas. We have grouped the various terms so that they match the summands which show up in \(\free\M\). Thus, for any $a$ and $b$ in $\M(H/H)$ we need to start with the element 
$$
N(a+b)-N(a)-N(b)
$$ 
in the free summand of $(\free\M)(G/G)$ and subtract elements in the transfer summand that are analogous to each $tr^G_{C_{p^j}}(N^{C_{p^j}}_{H}((\uab)_f))$ (which arise from the transfers of the other free summands) and to $tr^G_H(g_H(a,b))$ (which comes from the box powers of $\M$). 

We will first construct an element analogous to $(\uab)_f$. 

\begin{defn} 
Let $f\in \mathcal I_j/G$, and let $K=C_{p^j}$. The map $\{a,b\}\to\M(H/H)$ then defines an element
$$
(\ba\bb)_f=\prod_{gK\in G/K} f(gK)=f(eK)\times\dots\times f(\gamma^{|G/K|-1}K)\in \Map\big(G/K,\M(H/H)\big).
$$
Let $N((\ba\bb)_f)$ be the corresponding element in the free summand of $(\free\M)(G/K)$.
\end{defn}

By Fact \ref{monfacts} we can define an element in $(\free\M)(G/H)$ that is analogous to $g_H(a,b)$. 

\begin{defn}\label{defgab} Define $g_H(\ba,\bb)$ in $\M(H/H)^{\otimes |G/H|}$ by
$$
g_H(\ba,\bb) =  \sum_{f\in\mathcal I_k/G}\left(\bigotimes_{gH \in G/H}f(gH)\right) = \sum_{f\in\mathcal I_k/G}\left( f(eH) \otimes \cdots \otimes f(\g^{|G/H|-1}H)\right).
$$ 
\end{defn}

Thus, one type of generator of $(\tr\M)(G/G)$ is of the form

\begin{multline*}
N(a+b) - N(a) - N(b)  \\
- \sum_{H<C_{p^j}<G} tr^G_{C_{p^j}}\left(\sum_{f\in\mathcal I_{j}/G}N\left((\ba \bb)_f\right)\right) -  tr^G_H(g_H(\ba,\bb)).
\end{multline*}


Hence, when we quotient by $\tr\M$ to define $\ngh\M$ we have the following relation in $(\ngh\M)(G/G)$.

\begin{multline*}\label{TRsum}
N(a+b)  = \\    N(a) +  N(b) + \sum_{H<C_{p^j}<G} tr^G_{C_{p^j}}\left(\sum_{f\in\mathcal I_{j}/G}N\left((\ba\bb)_f\right)\right) +  tr^G_H(g_H(\ba,\bb))
\end{multline*}

Similarly, we want to mimic the Tambara reciprocity formula for transfers given in Corollary \ref{cor:TRtransfers};
\begin{eqnarray*}
N^G_Htr^H_{H'}(x) &=& tr^G_{H'}(f(x))
\end{eqnarray*}
where  $$f(x) = \sum_{s=1}^r \prod_{i=0}^{|G/H|-1} \g^i \g^{m_{i,s}} x.$$ 

\begin{defn}\label{deffx} Define $f(\bmx)$ in $\M(H/H')^{\otimes |G/H|}$  by  $$ f(\bmx)=\sum_{s=1}^r \g^{m_{0,s}}x \otimes \g^{m_{1,s}}x\otimes \cdots \otimes \g^{m_{|G/H|-1,s}}x.$$ So, $tr^G_{H'}\left(f(\bmx)\right)$ is in the transfer summand of $(\free\M)(G/G)$.
\end{defn}

Hence, the other type of generator of $(\tr\M)(G/G)$ is of the form 
$$N(tr^H_{H'}(x)) - tr^G_{H'}(f(\bmx)),$$ so that in $(\ngh\M)(G/G)$ we have the relation $$N( tr^H_{H'}(x)) = tr^G_{H'}(f(\bmx)),$$

We now define the Tambara reciprocity sub-Mackey functor $\tr\M$. For now it is an abuse of terminology to call $\tr\M$ a Mackey functor. However, we prove that it is in fact a Mackey functor in Theorem~\ref{TRisaMF}.
\begin{defn}[The Tambara Reciprocity Sub-Mackey Functor, $\tr\M$]\label{TRsubmods}
We define $\tr\M$ as follows. 

For $H' \le H$ define $(\tr\M)(G/H')=0$.
 
If $H<K<G$ then we inductively define $(\tr\M)(G/K)$ by
$$ 
(\tr\M)(G/K) =  \sum_{i=1}^{|G/K|} \left((\free[K]\M)^{\bx i-1} \bx \tr[K]\M \bx (\free[K]\M)^{\bx |G/K|-i}\right)(K/K)
$$
 
If $K=G$, 
then let $\tr[G]\M(G/G)$ be the subgroup generated by 
\begin{enumerate}
\item the ``addition relations'':
\begin{multline*}
N(a+b) - N(a) - N(b)  \\
- \sum_{H<C_{p^j}<G} tr^G_{C_{p^j}}\left(\sum_{f\in\mathcal I_{j}/G}N\left((\ba\bb)_f\right)\right) -  tr^G_H(g_H(\ba,\bb))
\end{multline*}
for all $a,b\in\M(H/H)$, 
\item the ``transfer relations'':
$$
N(tr^H_{H'}(x)) - tr^G_{H'}(f(\bmx))
$$
for all $x\in\M(H/H')$, and 
\item the transfer from smaller subgroups:
$$
Im\left(tr^G_{G'}((\tr\M)(G/G'))\right),
$$
where $G'$ is the maximal subgroup of $G$.
\end{enumerate}
\end{defn}

\begin{remark}
The values of $\tr[G]\M(G/K)$ with $K$ a proper subgroup are so that we have a copy of the Tambara reciprocity relations for each box factors in $i_K^\ast\free[G]\M$.
\end{remark}

We now show that $\tr\M$ is a Mackey functor. 

\begin{thm}\label{TRisaMF} The collection of submodules $\tr\M$ given in Definition \ref{TRsubmods} forms a subMackey functor of $\free\M$.
\end{thm}

\begin{proof}
We will use induction on the order of the group. The base case holds since for $H'\le H$, $(\tr\M)(G/H')$ is zero. Assume for induction that for all $K$ such that $H<K<G$, $\tr[K]\M$ is a Mackey functor.  Since each $\textsc{tr}^K\M$ is a Mackey functor for $H<K<G$, it follows that $\tr\M$ is well-defined. Further, by definition, $\tr\M$ is closed under the transfer maps and under the Weyl action, since this action is just the permutation action. Thus, it remains to show that $\tr\M$ is closed under the restriction maps. In particular, let $G'$ be the maximal subgroup of $G$. We will show that $res^G_{G'}((\tr\M)(G/G))$ is contained in $(\tr\M)(G/G')$.

From the formula 
$$
res_{G'}^G tr_{G'}^G(x)=\sum_{gG'\in G/G'} gx,
$$
we know that the restriction of the elements in the image of the transfer lands in the desired submodule. We need only check the ``addition'' generators and the ``transfer'' generators. These are very similar, combinatorial arguments, the main difficulty of which is in good, consistent notation. We spell the argument out carefully for the addition generators; the transfer case is similar enough that including it is unenlightening.

We now begin considering the ``addition relations'' 
\begin{multline*}
N(a+b) - N(a) - N(b)  \\
- \sum_{H<C_{p^j}<G} tr^G_{C_{p^j}}\left(\sum_{f\in\mathcal I_{j}/G}N\left((\ba \bb)_f\right)\right) -  tr^G_H(g_H(\ba,\bb)).
\end{multline*}
By definition, 
\begin{multline*}
\lefteqn{N(a) +  N(b) + \sum_{H<C_{p^j}<G} tr^G_{C_{p^j}}\left(\sum_{f\in\mathcal I_{j}/G} N\left((\ba \bb)_f\right)\right) 
+ tr^G_H\left(g_H(\ba,\bb)\right)}\\
= \sum_{f\in Map(G/H,\{a, b\})/_{G}} tr^G_{K_f}\left(\bigotimes_{i=0}^{|G/K_f|-1} f(\g^{i}K_f)\right)
\end{multline*} 
where each $K_f$ is the stabilizer of $f$. Thus, we need to show that 
\begin{eqnarray*}
res^G_{G'}(N(a+b))&=&res^G_{G'}\left(\sum_{f\in Map(G/H,\{a, b\})/_{G}} tr^G_{K_f}\left(\bigotimes_{i=0}^{|G/K_f|-1} f(\g^{i}K_f)\right)\right) \\
&=&\sum_{f\in Map(G/H,\{a, b\})/_{G'}}tr^{G'}_{K_f}\left(\bigotimes_{i=0}^{|G/K_f|-1} f(\g^{i}K_f)\right).
\end{eqnarray*}

By definition of the restriction in \(\free\M\), we have 
\[
res^G_{G'}(N(a+b)) = (N(a+b))^{\otimes |G/G'|},
\]
and by our induction hypothesis, in $(\tr\M)(G/G')$ we have 
\begin{eqnarray*}N\left(a+b\right)  &=&   N(a) +  N(b) + \sum_{H<C_{p^j}<G'} tr^{G'}_{C_{p^j}}\left(\sum_{f \in \mathcal{I}_j/G'} N\left((\ba \bb)_f\right)\right) 
+ tr^{G'}_H\left(g_H(\ba,\bb)\right) \\
&=& \sum_{f\in Map(G'/H,\{a, b\})/_{G'}} tr^{G'}_{K_f}\left(\bigotimes_{i=0}^{|G'/K_f|-1} f(\g^{ip}K_f)\right).
\end{eqnarray*}

Combining these, we have   
\begin{eqnarray*}\label{resN(a+b)}res^G_{G'}(N(a+b)) &=& \left(\sum_{f\in Map(G'/H,\{a, b\})/_{G'}} tr^{G'}_{K_f}\left(\bigotimes_{i=0}^{|G'/K_f|-1} f(\g^{ip}K_f)\right)\right)^{\otimes |G/G'|}.
\end{eqnarray*}
Therefore, we need to show the following equality.
\begin{multline}\label{eqn1}
    \sum_{f\in Map(G/H,\{a, b\})/_{G'}}tr^{G'}_{K_f}\left(\bigotimes_{i=0}^{|G/K_f|-1} f(\g^{i}K_f)\right)\\
    = \left(\sum_{f\in Map(G'/H,\{a, b\})/_{G'}} tr^{G'}_{K_f}\left(\bigotimes_{i=0}^{|G'/K_f|-1} f(\g^{ip}K_f)\right)\right)^{\otimes |G/G'|}
\end{multline}
We pause here to draw the reader's attention to the apparent asymmetry in the exponents of \(\gamma\) on the two sides of the equation. This is because since  \(\gamma\) is a chosen generator of \(G\), the element \(\gamma^p\) is then a chosen generator for \(G'\). The left-hand side is the restriction of a formula from \(G\), and hence uses \(\gamma\); the right-hand side is the corresponding formula for \(G'\) and hence uses its generator \(\gamma^p\).

Here the description in terms of the box product is most useful. The tensor product on the right-hand side of Equation~\ref{eqn1} distributes over the sum via the usual formula:

\begin{multline}\label{eqn:fhats}
\left(\sum_{f\in Map(G'/H,\{a, b\})/_{G'}} tr^{G'}_{K_f}\left(\bigotimes_{i=0}^{|G'/K_f|-1} f(\g^{ip}K_f)\right)\right)^{\otimes |G/G'|}= \\
\sum_{\vec{f}\in (Map(G'/H,\{a,b\})/{G'})^{\times p}} \bigotimes_{j=0}^{p-1} 
\left(tr_{K_{f_j}}^{G'}\left(\bigotimes_{i=0}^{|G'/K_{f_j}|-1} f_j(\gamma^{ip}K_{f_j})\right) \right),
\end{multline} 
where \(f_i\) is the \(i\)th factor of \(\vec{f}\).

Frobenius reciprocity allows us to rewrite the right-hand side of Equation~\ref{eqn:fhats}. For each \(\vec{f}\in \big(Map(G'/H,\{a,b\})/G'\big)^{\times p}\), let \(K_f\) be the intersection of the stabilizers \(K_{f_i}\). Since the subgroups are nested and since we have only \(p\) terms, \(K_f=K_{f_j}\) for some \(j\), and without loss of generality, we may assume that \(K_f=K_{f_0}\) (otherwise, we simply make the obvious bookkeeping change). Frobenius reciprocity shows that we have an equality
\begin{multline}\label{eqn:fhatFrob}
\bigotimes_{j=0}^{p-1} 
\left(tr_{K_{f_j}}^{G'}\left(\bigotimes_{i=0}^{|G'/K_{f_j}|-1} f_j(\gamma^{pi}K_{f_j})\right) \right)= \\ tr_{K_f}^{G'}\left(\left(\bigotimes_{i=0}^{|G'/K_{f_0}|-1} f_0(\gamma^{pi}K_{f_0})\right)\otimes  
\bigotimes_{j=1}^{p-1}\left(
\sum_{i=0}^{[G':K_{f_j}]-1}\gamma^{pi} res_{K_f}^{K_{f_j}} \left(\bigotimes_{k=0}^{|G'/K_{f_j}|-1} f_j(\gamma^{pk}K_{f_j})\right)
\right)\right). 
\end{multline}
By construction, 
\[
res_{K_f}^{K_j}\left(\bigotimes_{k=0}^{|G'/K_{f_j}|-1} f_j(\gamma^{pk}K_{f_j})\right)=
\bigotimes_{k=0}^{|G'/K_f|-1} f_j(\gamma^{pk}K_f),
\]
and the action of \(\gamma^{pi}\) simply rotates the tensor factors as tensor induction.

We can further simplify the right-hand side of Equation~\ref{eqn:fhatFrob} by noticing that there is a coordinate-wise action of \({G'}^{\times p}\) on \(\Map(G'/H,\{a,b\})^{\times p}\). We see that in the \(i\)th factor of our formula, we are summing over the orbit \(G'/K_{f_i}\). There is an apparent asymmetry here too, because the first factor appears in isolation. However, if we remember that the transfer factors through the Weyl action then we notice that the sum can be taken over the set
\[
\mathcal G_{\vec{f}}=\left(\prod_{i=0}^{p-1} G'/K_{f_i}\right)/G'.
\]
If \(\vec{g}\in\mathcal G_{\vec{f}}\), then let \(g_i\) be the \(i\)th coordinate of the representative where the first coordinate is \(1\). We can then rewrite the right-hand side of Equation~\ref{eqn:fhatFrob} as
\begin{equation}
    \sum_{\vec{g}\in \mathcal G'} tr_{K_f}^{G'}
    \bigotimes_{i=0}^{p-1} \left(g_i\bigotimes_{j=0}^{|G'/K_f|-1} f_i(\gamma^{pj}K_f)\right).
\end{equation}
This gives us a refined version of what we need to show, so now we need to show
\begin{multline}\label{eqn:Desired}
    \sum_{f\in Map(G/H,\{a, b\})/_{G'}}tr^{G'}_{K_f}\left(\bigotimes_{i=0}^{|G/K_f|-1} f(\g^{i}K_f)\right) = \\
    \sum_{\vec{f}\in (Map(G'/H,\{a,b\})/{G'})^{\times p}}    \sum_{\vec{g}\in \mathcal G'} tr_{K_f}^{G'}
    \bigotimes_{i=0}^{p-1} \left(g_i\bigotimes_{j=0}^{|G'/K_f|-1} f_i(\gamma^{pj}K_f)\right).
\end{multline}

The argument now is combinatorial. It is helpful to think about the indexing set
$$
Map(G/H,\{a,b\}),
$$
both as a $G$-set and as a $G'$-set. First, the cosets $G/G'$ partition $G/H$ into $p$ copies of $G'/H$, indexed as $G'/H$, $\gamma G'/H$, $\dots$, $\gamma^{p-1}G'/H$. This partitioning is not $G$-equivariant, but since $G$ is abelian, it is $G'$-equivariant, and this gives us a $G'$-equivariant isomorphism
\[
G/H\cong \{0,\dots,p-1\}\times G'/H.
\]
Using this isomorphism, we have \(G'\)-equivariant isomorphisms
\begin{multline*}
Map(G/H,\{a,b\})\cong \\ Map\big(\{0,\dots,p-1\}\times G'/H,\{a,b\}\big)\cong  \big(Map(G'/H,\{a,b\})\big)^{\times p}.
\end{multline*}
Thus, given a function \(f\colon G/H\to\{a,b\}\), these isomorphisms allows us to write $f$ as 
\[
f\leftrightarrow \prod_{i=0}^{p-1} f_i=(f_0,\dots,f_{p-1})=:\vec{f},
\]
where \(f_i\) is the restriction of \(f\) to the coset \(\gamma^i G'/H\).

Although the isomorphisms used are not \(G\)-equivariant, we can still  describe the action of \(\gamma\) on this product:
\[
\gamma (f_0,\dots,f_{p-1})=(f_1,\dots,f_{p-1},\gamma^{p}f_0).
\]
We can use this to determine the stabilizer of \(f\). We know that the only \(G\)-fixed functions are the constant ones, and here, that means that each \(f_i\) is constant and they all agree. The \(G'\)-fixed functions come from those sequences \(\vec{f}\), where each \(f_i\) is a constant function. Finally, for a non-constant function, the stabilizer of \(f=(f_0,\dots,f_{p-1})\) is just
\[
K_f=Stab(f)=\bigcap_{i=0}^{p-1} K_{f_i},
\]
where \(K_{f_i}\) is the stabilizer of \(f_i\).

The stabilizer of \(\vec{f}\) in \((G')^{\times p}\) is 
\[
K_{\vec{f}}=K_{f_0}\times\dots\times K_{f_{p-1}}.
\]
This means that using the coordinate-wise action of \(G'\), each (non-constant) \(p\)-tuple \(\vec{f}\) contributes 
\[
|(G')^{\times p}/K_{\vec{f}}|=\prod_{i=0}^{p-1}[G':K_{f_i}]
\]
distinct functions to \(Map(G/H,\{a,b\})\). Putting this together, we see that the two sides of Equation~\ref{eqn:Desired} are giving the two ways to express this sum over \(Map(G/H,\{a,b\})\), completing the proof.
\end{proof}


\begin{defn}\label{nghdef}Let $M$ be an $H$-Mackey functor. Define $\ngh\M$ by $$\ngh\M := \free\M/\tr\M.$$
\end{defn}

\begin{lemma}\label{lem:nghrestriction} For $H<K<G$, $i_K^* \ngh\M$ is isomorphic to $(N^K_H\M)^{\bx |G/K|}$. 
\end{lemma}

\begin{proof} This result follows directly from Lemma~\ref{freegiso} and from the fact that the box powers of a quotient is the quotient of box powers.
\end{proof}

\section{Proof of the Main Theorem}

In this section we verify that the construction $\ngh\M$ satisfies the Main Theorem. Recall that a morphism $\phi\co\M \to \lu$ of Mackey functors consists of a collection of homomorphisms $\{\phi_H\co \M(G/H) \to \lu(G/H) : H \le G\}$ that commute with the appropriate restriction and transfer maps.  

\begin{thm}\label{is_a_functor} For all subgroups $H$ of $G$, the map $\ngh\co \mackh \to \mackg$ given by $\M \mapsto \ngh\M$ is a functor.
\end{thm}

\begin{proof} Given a morphism $\phi\co \M \to \lu$ in $\mackh$ we define the associated morphism $\ngh(\phi)\co \ngh\M \to \ngh\lu$ in $\mackg$ as follows.

For all subgroups $H'$ in $H$ define $\ngh(\phi)_{H'}$ to be $\phi_{H'}^{\bx |G/H|}$. If $K=C_{p^j}$ is a subgroup such that $H<K\le G$, then we inductively define $\ngh(\phi)_K$ so that it is compatible with the appropriate restriction and transfer maps. More specifically, for all $tr(x)$ in the $Im(tr^j_{j-1})$ summand of $(\ngh\M)(G/K)$ we define $\ngh(\phi)_K(tr(x))$ to be 
$$tr_{j-1}^j(\ngh(\phi)_{C_{p^{j-1}}}(x)).$$ 
If $N(\bfm^{\times|G/K|})$ is a generator in the image of the free summand in $(\ngh\M)(G/K)$, define $\ngh(\phi)_K(N(\bfm^{\times|G/K|}))$ to be $$N(\phi_H(m_e) \times \phi_H(m_{\g}) \times \cdots \times \phi_H(m_{\g^{|G/K|-1}})).$$  Further, $\ngh(\phi)$ maps $\tr\M$ to $\tr\lu$ since each $(\ba\bb)_f$, $g_H(\ba,\bb)$ and $f(\bmx)$ are universally determined by the group $G$. Then, by definition, the maps $\{\ngh(\phi)_K : K \le G\}$ form a natural transformation of $G$-Mackey functors, and the assignment $\phi \mapsto \ngh(\phi)$ is functorial. 
\end{proof}

\subsection{$\ngh$ is Isomorphic to the Composition $N^G_KN^K_H$}

We next prove that the norm functors satisfy Property (a) of the Main Theorem. Thus, we will prove that $\ngh\co \mackh \to \mackg$  (given by $\M \mapsto \ngh\M$) is isomorphic to the composition of functors $N^G_KN^K_H$ whenever $H<K<G$.  By induction it suffices to show that for $G'$ maximal in $G$, $N^G_{G'}N^{G'}_H\M$ is isomorphic to $\ngh\M$.

\begin{thm} Let $H$ be a subgroup of $G$ and let $G'$ be the maximal subgroup of $G$. Then $N^G_{G'}N^{G'}_H\M$ is isomorphic to $\ngh\M$.
\end{thm}

\begin{proof} First, we will show that $(N^G_{G'}N^{G'}_H\M)(G/K)$ is isomorphic to $(\ngh\M)(G/K)$ for all $K \le G'$. By Lemma~\ref{lem:nghrestriction} $(\ngh\M)(G/K)$ is isomorphic to $(N^K_H\M)^{\bx|G/K|}(K/K)$. Further, by definition, $$(N^G_{G'}N^{G'}_H\M)(G/K) = (N^{G'}_H\M)^{\bx |G/G'|}(G'/K),$$ and by Lemma~\ref{lem:nghrestriction} $(N^{G'}_H\M)(G'/K)$ is isomorphic to $(N^K_H\M)^{\bx|G'/K|}$. Hence, $(N^G_{G'}N^{G'}_H\M)(G/K)$ is isomorphic to $\left((N^K_H\M)^{\bx|G'/K|}\right)^{\bx|G/G'|}(K/K)$, which is isomorphic to $(N^K_H\M)^{\bx|G/K|}(K/K)$.

It remains to show that $(N^G_{G'}N^{G'}_H\M)(G/G)$ is isomorphic to $(\ngh\M)(G/G)$. By definition,
\begin{eqnarray*}
(\ngh\M)(G/G) &=& \left(\free\M\right)(G/G)/(\tr\M)(G/G) \\
&=& \left(\Z\{\M(H/H)\} \oplus (\ngh\M)(G/G')/_{W_G(G')}\right)/_{(\tr\M)(G/G)}.
\end{eqnarray*}
Then the module $(N^G_{G'}N^{G'}_H \M)(G/G)$ is $$\left( \Z\{(N^{G'}_H \M)(G'/G') \} \oplus (N^G_{G'}N^{G'}_H\M)(G/G')/_{W_G(G')}\right)/_{(\tr\M)(G/G)},$$ which equals $$\Big(\Z\{[\Z\{\M(H/H)\} \oplus Im(tr)]/_{(\tr[{G'}]\M)(G'/G')}\} \oplus Im(tr^G_{G'}) \Big)/_{(\tr\M)(G/G)}.$$ Quotienting by $(\tr\M)(G/G)$ identifies all elements in the transfer summand of $(N^{G'}_H \M)(G'/G')$ with elements in $Im(tr^G_{G'})$. Thus,
\begin{eqnarray*}
(N^G_{G'}N^{G'}_H\M)(G/G) &\cong& 
\left(\Z\{\Z\{\M(H/H)\}\} \oplus Im(tr^G_{G'}) \right)/_{(TR^G\M)(G/G)}. 
\end{eqnarray*}
Finally, since the Tambara reciprocity submodule identifies sums of generators in the free summand with elements in the transfer summand, it follows that $(N^G_{G'}N^{G'}_H\M)(G/G)$ is isomorphic to $\left(\Z\{\M(H/H)\} \oplus Im(tr^G_{G'})\right)/_{(\tr\M)(G/G)}$. Since $Im(tr^G_{G'})$ is isomorphic to $(\ngh\M)(G/G')/_{W_G(G')}$, it follows that $$(N^G_{G'}N^{G'}_H\M)(G/G)\cong(\ngh\M)(G/G).$$ \end{proof}






\begin{cor}\label{composable} For all subgroups $H$ and $K$ of $G$ such that $H<K<G$, the norm functor $N^G_H: \mackh \to \mackg$ is isomorphic to the composition of functors $N^G_KN^K_H$.
\end{cor}

\subsection{The Norm Functors are Strong Symmetric Monoidal}

We now show that for all subgroups $H$ of $G$, $\ngh\co \mackh \to \mackg$ is strong symmetric monoidal, and thus the norm functors satisfy Property (b) of the Main Theorem. By Corollary \ref{composable}, it suffices to let $H$ be maximal in $G$. In this case $\ngh\M$ simplifies nicely. Indeed, if $H'$ is a subgroup of $H$, then $(\ngh\M)(G/H') = \M^{\bx |G/H|}(H/H')$. The only remaining module is $(\ngh\M)(G/G)$.
 
 \begin{fact}\label{maxglevel} The module $(\ngh \M)(G/G)$ is $(\free \M)(G/G)/(\tr \M)(G/G)$, which is  $$\left(\Z\{\M(H/H)\} \oplus \M^{\bx|G/H|}(H/H)/_{W_G(H)}\right)/_{TR}, $$ where $TR$ is the ``Tambara reciprocity'' submodule generated by elements of the forms
 $$N(a+b) -N(a) -N(b) -tr^G_H(g_H({\bm{a}},{\bm{b}}))$$ and $$N(tr^H_{H'}(x)) - tr^G_{H'}(f({\bm{x}}))$$ for all $a$ and $b$ in $\M(H/H)$, $x$ in $\M(H/H')$ and $H'<H$. 
 \end{fact}
 
 To show that $\ngh$ is strong symmetric monoidal, given $H$-Mackey functors $\M$ and $\lu$ we will build an isomorphism $$\Psi \co \ngh\M \bx \ngh\lu \to \ngh(\M \bx \lu)$$ by defining a collection of isomorphisms $$\{\Psi_K\colon (\ngh\M \bx \ngh\lu)(G/K) \to \ngh(\M \bx \lu)(G/K) \text{ for all } K \le G\}.$$ For $H' \le H$ the isomorphism $\Psi_{H'}$ will be induced from properties of the box product. So, the work lies in defining $\Psi_G$. Before we do so, we explicitly describe $(\ngh\M\bx\ngh\lu)(G/G)$ and $\ngh(\M \bx \lu)(G/G)$.
 
 \begin{lemma}\label{explnmnl} Let $H$ be maximal in $G$ and let $\M$ and $\lu$ be $H$-Mackey functors. 
 The module $(\ngh\M\bx \ngh\lu)(G/G)$ is isomorphic to $$ \left( \Z\{\M(H/H) \times \lu(H/H)\} \oplus Im(tr^G_H)\right)/_{\widetilde{FR}},$$ where $Im(tr^G_H)$ is $(\M^{\bx|G/H|} \bx \lu^{\bx|G/H|})(H/H)/_{W_G(H)}$. The submodule $\widetilde{FR}$ is generated by $$N((a+b)\times l) - N(a\times l) - N(b \times l)-tr^G_H\Big(g_H({\bm{a}},{\bm{b}})\otimes \botimes_{|G/H|}l\Big),$$ $$N(m \times (y+z)) - N(m\times y) - N(m\times z) - tr^G_H\Big(\big(\botimes_{|G/H|} m\big) \otimes g_H({\bm{y}},{\bm{z}})\Big),$$ $$ N(tr^H_{H'}(d)\times l) - tr^G_{H'}\Big(f({\bm{d}})\otimes \botimes_{|G/H|} res^H_{H'}(l)\Big),$$ and $$N(m\times tr^H_{H'}(x)) - tr^G_{H'}\Big(\botimes_{|G/H|} res^H_{H'}(m) \otimes f({\bm{x}})\Big)$$ for all $a$, $b$, and $m$ in $\M(H/H)$, $d$ in $\M(H/H')$, $y$ $z$, and $l$ in $\lu(H/H)$, $x$ in $\lu(H/H')$ and subgroups $H'$ of $H$. \end{lemma}
 
 \begin{proof} First, $$(\ngh\M \bx \ngh\lu)(G/G) = \left(\ngh\M(G/G) \otimes \ngh\lu(G/G) \oplus Im(tr^G_H) \right)/_{FR}.$$ The submodule $\widetilde{FR}$ stems from combining the relations defined by the $FR$ submodule with the relations from each $TR$ submodule. For example, by the $TR$ submodule of $\ngh\M(G/G)$ we have $$N(a+b) \otimes N(l) - N(a) \otimes N(l) - N(b) \otimes N(l) - tr^G_H(g_H(\bm{a},\bm{b})
)\otimes N(l).$$ But, now we use the $FR$ submodule and the fact that $res^G_H(N(l)) = \bigotimes_{|G/H|} l$ to identify $tr^G_H(g_H(\bm{a},\bm{b}))\otimes N(l)$ with $tr^G_H\left(g_H(\bm{a},\bm{b})\otimes \bigotimes_{|G/H|} l\right)$. Finally, we use the fact that for any modules $R$ and $S$, $\Z\{R\} \otimes \Z\{S\}$ is isomorphic to $\Z\{R \times S\}$ to arrive at the module given in the lemma.
 \end{proof}
 
We introduce some notation before describing $\ngh(\M \bx \lu)$.  First, let $g_H({\bm{a}}\otimes {\bm{y}},{\bm{b}}\otimes {\bm{z}})$ be a polynomial similar to $g_H({\bm{a}},{\bm{b}})$. Specifically, let $$g_H({\bm{a}}\otimes {\bm{y}},{\bm{b}}\otimes {\bm{z}}) = \sum_{f \in \mathcal I_k/G} \left(\bigotimes_{gH \in G/H} f(g H)\right)$$ where now $\mathcal I_k = \left( Map(G/H, \{a\otimes y,b\otimes z\}) - Map(G/G, \{a \otimes y,b\otimes z\})\right)$.

Further, we define $f({\bm{d}}\otimes res^H_{H'}({\bm{y})})$ similarly to $f({\bm{x}})$. So, $$ f({\bm{d}}\otimes res^H_{H'}({\bm{y})}) = \sum_{s=1}^r \botimes_{i=0}^{|G/H|-1} \left(\g^{m_{i,s}} d\otimes  res^H_{H'}(y)\right).$$ 
We define $f(res^H_{H'}({\bm{a}})\otimes {\bm{x}})$  analogously.

 \begin{lemma}\label{explnml} Let $H$ be maximal in $G$ and let $\M$ and $\lu$ be $H$-Mackey functors. The module $\nml(G/G)$ is isomorphic to $$ \left(\Z\{\M(H/H) \otimes \lu(H/H)\} \oplus Im(tr^G_H) \right)/_{\widetilde{TR}},$$ where $Im(tr^G_H)$ is $(\M\bx \lu)^{\bx|G/H|}(H/H)/_{W_G(H)}$, and $\widetilde{TR}$ is generated by the following elements for all $a\otimes y$ and $b\otimes z$ in  $\M(H/H) \otimes \lu(H/H)$, $d$ in $\M(H/H')$, $x$ in $\lu(H/H')$, and subgroups $H'$ of $H$: $$N(a\otimes y + b\otimes z) - N(a\otimes y) - N(b \otimes z) - tr^G_H(g_H({\bm{a}}\otimes {\bm{y}},{\bm{b}}\otimes {\bm{z}})),$$ $$ N(tr^H_{H'}(d)\otimes y) - tr^G_{H'}(f({\bm{d}}\otimes res^H_{H'}({\bm{y}}))),$$  $$N(a\otimes tr^H_{H'}(x)) - tr^G_{H'}(f(res^H_{H'}({\bm{a}})\otimes {\bm{x}})).$$  \end{lemma}
 
 \begin{proof} The module $$\ngh(\M \bx \lu)(G/G) = \left(\Z\{(\M\bx \lu)(H/H)\} \oplus Im(tr^G_H)\right)/_{TR},$$ and by Definition \ref{bxproddef}, $(\M\bx \lu)(H/H)= (\M(H/H)\otimes \lu(H/H) \oplus Im(tr))/_{FR}$. Thus, since we quotient by the $TR$ submodule in $\ngh(\M\bx\lu)(G/G)$, we can identify all elements in $Im(tr)$ with the appropriate elements in $Im(tr^G_H)$. If we then combine the Frobenius reciprocity submodule of  $(\M\bx\lu)(H/H)$ with Tambara reciprocity the resulting module is isomorphic to $$ \left(\Z\{\M(H/H) \otimes \lu(H/H)\} \oplus Im(tr^G_H) \right)/_{\widetilde{TR}}.\vspace{-20pt}$$\end{proof}

 \begin{thm}\label{nghsm} For all subgroups $H$ of $G$, the norm functor $\ngh\co \mackh \to \mackg$ is a strong symmetric monoidal functor.
\end{thm}

\begin{proof} By Corollary \ref{composable} it suffices to let $H$ be maximal  in $G$.  We will build an isomorphism $\Psi\co \ngh \M \bx \ngh \lu \to \ngh(\M \bx \lu)$ by defining a collection of isomorphisms $$\{\Psi_K: (\nmnl)(G/K) \to \nml(G/K) \text{ for all } K\le G\}.$$

First, let $H'$ be a subgroup of $H$. By definition, $(\nmnl)(G/H')= (\M^{\bx |G/H|} \bx \lu^{\bx|G/H|})(H/H')$ and $\nml(G/H')=(\M \bx \lu)^{\bx |G/H|}(H/H')$. Since $\bx$ is symmetric monoidal we define $\Psi_{H'}$ to be the natural isomorphism $$ (\M^{\bx |G/H|} \bx \lu^{\bx|G/H|})(H/H') \to (\M \bx \lu)^{\bx |G/H|}(H/H').$$ 

It remains to  define the isomorphism  $\Psi_G$.  By Lemmas \ref{explnmnl} and \ref{explnml},  $$(\nmnl)(G/G) \cong \left( \Z\{\M(H/H) \times \lu(H/H)\} \oplus Im(tr^G_H)\right)/_{\widetilde{FR}},$$ and $$\nml(G/G) \cong \left(\Z\{\M(H/H) \otimes \lu(H/H)\} \oplus Im(tr^G_H) \right)/_{\widetilde{TR}}.$$

Thus, we define $\Psi_G$ to be the direct sum of the following two maps. First, let $\psi_G\colon Im(tr^G_H) \to Im(tr^G_H)$ be the isomorphism induced from $\Psi_H$. Then define $$\psi_{G}'\co\Z\{\M(H/H) \times \lu(H/H)\} \to \Z\{\M(H/H) \otimes \lu(H/H)\}$$ by  $\psi'_G(N(a\times y)) =N(a\otimes y)$, and so we define $\Psi_G$ to be $\psi_G' \oplus \psi_G$. Because of the way $\psi_G$ reindexes the elements of $Im(tr^G_H)$, the map $\Psi_G$ sends all elements in $\widetilde{FR}$ to elements in $\widetilde{TR}$. Further, by construction, $\Psi_G tr^G_H = tr^G_H\Psi_H$ and $res^G_H \Psi_G = \Psi_H res^G_H$.   Finally, the relations defined by $\widetilde{FR}$ are analogous to those that define a tensor product from a Cartesian product. It follows that $\Psi_G$ is an isomorphism.
\end{proof}

 We end this section with a proof of the Main Theorem.
  \begin{proof}[Proof of the Main Theorem] For all subgroups $H$ of $G$ and $H$-Mackey functors $\M$, let $\ngh\M$ be the $G$-Mackey functor defined in Definition \ref{nghdef}. Then define the norm functors $\ngh\co \mackh \to \mackg$ by $\M \mapsto \ngh\M$. These maps satisfy all properties given in the Main Theorem by Theorem \ref{is_a_functor}, Corollary  \ref{composable}, and Theorem \ref{nghsm}.
  \end{proof}
  
\section{Proof of Theorem \ref{mainthm2}}


Finally, given a cyclic $p$-group $G$, we use the norm functors $\ngh:\mackh \to \mackg$  to define a $G$-symmetric monoidal structure on the category of $G$-Mackey functors. We will then show that $G$-Tambara functors are the $G$-commutative monoids under this structure, thus proving Theorem \ref{mainthm2}.

We begin with Hill and Hopkins' definition of a $G$-symmetric monoidal structure \cite{HillHopkins}. 

\begin{defn}\label{gsm} 
Let $\setiso$ be the category whose objects are finite $G$-sets  and whose morphisms are isomorphisms of $G$-sets. Further, let $(\C, \boxtimes, e)$ be a symmetric monoidal category. A {\emph{$G$-symmetric monoidal structure}} on $\C$ consists of a functor $$(-) \otimes (-)\co \setiso \times \C \to \C$$  that satisfies the following properties.
\begin{enumerate}
\item $(X \amalg Y) \otimes C$ is naturally isomorphic to  $(X \otimes C) \boxtimes (Y \otimes C)$ and $X \otimes (C \boxtimes D)$ is naturally isomorphic to $(X \otimes C) \boxtimes (X \otimes D)$.
\item When restricted to $\setisor \times \C$ this functor is the canonical exponentiation map given by $X\otimes C = C^{\boxtimes |X|}$.  
\item $X \otimes (Y \otimes C)$ is naturally isomorphic to $(X \times Y) \otimes C$.
\end{enumerate}
\end{defn}

\begin{thm}\label{mainthm}
 Let $\ih\co \mackg \to \mackh$ be the forgetful functor. For a cyclic $p$-group $G$ the functor $(-) \otimes (-)\co \setiso \times \mackg \to \mackg$ defined by  \begin{itemize}
\item $\emptyset \otimes \M := \A,$ where $\A$ is the Burnside Mackey functor,
\item $G/H \otimes \M := \ngh \ih \M$ for all orbits $G/H$ of $G,$ and 
\item $(X \amalg Y) \otimes \M := (X \otimes \M) \bx (Y \otimes \M)$ for all $X$ and $Y$ in $\setiso$
\end{itemize} is a $G$-symmetric monoidal structure on $\mackg.$ \end{thm}

\begin{proof}
Let $\M$ be a $G$-Mackey functor. The above functor $(-)\otimes (-)\co \setiso \times \mackg \to \mackg$ satisfies Property 1 of Definition \ref{gsm} because the norm functors $\ngh$ are strong symmetric monoidal for all subgroups $H$ of $G$.  Further, if $X$ is a finite set, then we can regard it as a disjoint union of $|X|$-many copies of the $G$-orbit $G/G$. Thus, $(-)\otimes (-)$ satisfies Property 2 of Definition \ref{gsm} since
$$X\otimes \M = (G/G\otimes \M)^{\bx|X|} = \M^{\bx|X|}.$$

Finally, to show that Property 3 of Definition \ref{gsm} holds it  suffices to show that $(G/K \times G/H) \otimes \M \cong G/K \otimes (G/H \otimes \M)$ for all orbits $G/H$ and $G/K$ of $G$. We first assume that $H$ is a subgroup of $K$, so $G/K \times G/H$ is isomorphic to $\amalg_{|G/K|} G/H$. Then
$$ (G/K \times G/H)\otimes \M \cong 
  (G/H \otimes \M)^{\bx |G/K|} \cong  N^G_H i_H^*(\M^{\bx |G/K|}).$$ On the other hand, $G/K \otimes (G/H \otimes \M) = \ngk \ik \ngh i_H^* \M$, and using Lemma \ref{lem:nghrestriction},  $\ik\ngh\ih\M$ is isomorphic to $(N^K_H\ih\M)^{\bx|G/K|}$. Then via Theorem \ref{nghsm}, $(N^K_H\ih\M)^{\bx|G/K|}$ is isomorphic to $N^K_H\ih\left(\M^{\bx|G/K|}\right)$, and therefore,
$$G/K\otimes (G/H\times \M) \cong \ngk N^K_H\ih\left(\M^{\bx|G/K|}\right) \cong  \ngh\ih\left(\M^{\bx|G/H|}\right). $$

Next, if $K \le H$ then $(G/K \times G/H) \otimes \M \cong N^G_K i_K^*\left( \M^{\bx |G/H|}\right)$. Moreover, $$G/K \otimes (G/H \otimes \M) = N^G_Ki_K^* \ngh \ih \M.$$ But, $i_K^* \ngh \ih \M = i_K^* \M^{\bx |G/H|}$, and so $G/K \otimes (G/H \otimes \M)= N^G_K i_K^*(\M^{\bx|G/H|})$ as well.
\end{proof}

 To define the commutative ring objects under a $G$-symmetric monoidal structure let $\C$ be a symmetric monoidal category with a $G$-symmetric monoidal structure $(-) \otimes (-)$. Every object $C$ in $\C$ defines a functor $$(-) \otimes C\co \setiso \to \C.$$ 

\begin{defn}\label{gcm} \cite{HillHopkins}  A {\emph{$G$-commutative monoid}} is an object $C$ in $\C$ together with an extension of $(-) \otimes C$ as given below.
$$\xymatrix{ \setiso \ar[r]^-{(-)\otimes C} \ar[d] & \C \\ \set \ar@{-->}[ur]}$$
\end{defn}

We will finish proving Theorem \ref{mainthm2} by showing that if we endow $\mackg$ with the $G$-symmetric monoidal structure defined in Theorem \ref{mainthm}, then a Mackey functor $\M$ is a Tambara functor if and only if it is a $G$-commutative monoid. We start by proving the forward implication in Proposition \ref{TFisGCM} and leave the reverse implication to Proposition \ref{gcmtotf}.

\begin{prop}\label{TFisGCM} Let $\M$ be a $G$-Mackey functor and endow $\mackg$ with the $G$-symmetric monoidal structure of Theorem \ref{mainthm}. If $\M$ has the structure of a Tambara functor, then $\M$ is a $G$-commutative monoid.\end{prop}

 To prove Proposition \ref{TFisGCM} we need to show that if $\su$ is a Tambara functor, then a map $X \to Y$ of $G$-sets induces a map $X \otimes \su \to Y \otimes \su$ of Mackey functors. Thus, we will extend the norm functors $\ngh \co \mackh \to \mackg$ to functors $\nngh \co \tambh \to \tambg$ on Tambara functors and show that $\nngh$ is left adjoint to the forgetful functor $\ih \co \tambg \to \tambh$. Then given a Tambara functor $\su$ we will use the counit of the above adjuction along with properties of finite $G$-sets and the fact that the box product is the coproduct in $\tambg$ (\cite{Strickland}) to induce a map $X \otimes \su \to Y \otimes \su$.


\begin{lemma} For all subgroups $H$ of $G$ the functor $\ngh\co \mackh \to \mackg$ extends to a functor $\nngh\co \tambh \to \tambg$. 
\end{lemma}

\begin{proof} We need to show that for all subgroups $H$ of $G,$ if $\su$ is an $H$-Tambara functor then $\ngh\su$ is a $G$-Tambara functor. However, by Corollary \ref{composable} it suffices to let $H$ be the maximal subgroup in $G$.  Since $\ngh$ is strong symmetric monoidal it naturally extends to a functor $\greenh \to \greeng$ where $\greeng$ is the category of $G$-Green functors. Hence, it remains to define the internal norm maps $N^K_{K'}\co (\ngh \su)(G/K') \to (\ngh \su)(G/K)$ for all subgroups $K'<K$ in $G$. 

Since the box product is the coproduct in $\tambh$  \cite[Prop 9.1]{Strickland}, if both $H'$ and $H''$ are subgroups of $H$ with $H''<H'$, then we define $N^{H'}_{H''}$ to be the $|G/H|$-fold box product of the  norm map  $N^{H'}_{H''}$ in $\su$. Lastly, we must define the norm $\ngh\co \su^{\bx|G/H|}(H/H) \to (\ngh\su)(G/G)$. This norm is the composition of the multiplication map of $\su$ with the map $N \co \su(H/H) \to (\ngh\su)(G/G)$ defined in Remark \ref{DefofN}.
Thus, letting $\mu\co \su^{\bx |G/H|} \to \su$ be the multiplication map of $\su$, we define $\ngh$ to be the composition $$\su^{\bx|G/H|}(H/H) \xrightarrow{\mu}  \su(H/H) \xrightarrow{N} (\ngh \su)(G/G).$$  This composition satisfies Properties 1, 2, 4, and 5 of Definition \ref{axiomatictfdef} by construction of the functor $\ngh\co \mackh \to \mackg$. (Property 3 is Frobenius Reciprocity.) It remains to show that this norm map satisfies Tambara Reciprocity (Property 6 of Definition \ref{axiomatictfdef}). 

First, we will show that $\ngh$ satisfies Tambara Reciprocity for sums (Corollary \ref{cor:TRsums}). Let $\botimes a_j = a_0 \otimes \cdots \otimes a_{|G/H|-1}$ and $\botimes b_j = b_0 \otimes \cdots \otimes b_{|G/H|-1}$ be simple tensors in $\su^{\bx |G/H|}(H/H)$. We need to show that $$\ngh(\botimes a_j + \botimes b_j) = \ngh(\botimes a_j) + \ngh(\botimes b_j) + tr^G_H(g_H(\botimes a_j,\botimes b_j)),$$ where $$g_H(\botimes a_j,\botimes b_j) = \sum_{f \in \mathcal I_k/G} \left(\prod_{gH\in G/H} gf(gH)\right) =\sum_{f \in \mathcal I_k/G} \left(\prod_{i=0}^{|G/H|-1} \g^i f(\g^i H)\right),$$ and $$\mathcal I_k = \left(Map\left(G/H, \left\{\botimes a_j, \botimes b_j\right\}\right) - Map\left(G/G, \left\{\botimes a_j, \botimes b_j\right\}\right)\right).$$ 

Now, in $\nngh \su$, 
\begin{eqnarray*}
\ngh\left(\botimes a_j + \botimes b_j\right) &=& 
 (N\circ \mu) \left(\botimes a_j + \botimes b_j\right) \\ &=& 
  N(a_0a_{1}\cdots a_{|G/H|-1} + b_0b_{1}\cdots b_{|G/H|-1}).
  \end{eqnarray*}
Let $a_0a_{1}\cdots a_{|G/H|-1}=\prod a_j $ and $b_0b_{1}\cdots b_{|G/H|-1}=\prod b_j $. Since we quotient by the Tambara reciprocity submodule $TR$ in $(\nngh\su)(G/G)$, it follows that 
$$N\left(\prod a_j + \prod b_j \right) = N\left(\prod a_j\right) + N\left(\prod b_j\right) + tr^G_H\left(g_H\left({\bm{\prod a_j}},{\bm{\prod b_j}}\right)\right),$$ and $g_H\left({\bm{\prod a_j}},{\bm{\prod b_j}}\right) = \sum_{f \in \mathcal I_k/G} \left(\botimes_{i=0}^{|G/H|-1}  f(\g^i H)\right)$ where now $$\mathcal I_k = \left(Map\left(G/H, \left\{\prod a_j, \prod b_j\right\}\right) - Map\left(G/G, \left\{\prod a_j, \prod b_j\right\}\right)\right).$$  

Further, since in $(\nngh \M)(G/G)$, $tr^G_H(g_H\left({\bm{\prod a_j}},{\bm{\prod b_j}}\right)$ lies in $\M^{\bx |G/H|}(H/H)/_{W_G(H)}$, it follows that we can write each $\botimes_{i=0}^{|G/H|-1}  f(\g^i H)$ as $$f(H)f(\g H) \cdots f(\g^{|G/H|-1}H)\otimes 1^{\otimes |G/H|-1}.$$ Therefore, $tr^G_H\left(g_H\left({\bm{\prod a_j}},{\bm{\prod b_j}}\right)\right)=tr^G_H(g_H(\botimes a_j,\botimes b_j))$, and so the norm map satisfies Tambara reciprocity for sums.

If we employ a strategy analogous to the one above we can show that the norm $\ngh$ in $\nngh \su$ also satisfies Tambara reciprocity for transfers (Corollary \ref{cor:TRtransfers}). Specifically, because we quotient by the $TR$ submodule in $(\nngh \su)(G/G)$ and by the appropriate Weyl action in the image of each transfer map, it follows that $$\ngh tr^H_{H'}(x) = N(tr^H_{H'}(\mu(x)) = tr^G_{H'}(f({\bm{\mu(x)}}) = tr^G_{H'}(f(x)). $$
Therefore, $\nngh \su$ is a Tambara functor.
\end{proof}

\begin{lemma}\label{tfadj} The functor $\nngh\co \tambh \to \tambg$ is left adjoint to the restriction functor $\ih\co \tambg \to \tambh$.
\end{lemma}

\begin{proof} Since we can compose adjunctions in a natural fashion \cite{maclane}, by Corollary \ref{composable}, it suffices to let $H$ be maximal in $G$. Let $\R$ be in $\tambg$ and $\su$ be in $\tambh$. Further, let $\tambh(\su,\ih\R)$ be the set of morphisms from $\su$ to $\ih\R$ in $\tambh$. We will show that $ \tambh(\su, \ih\R)$ is in natural bijective correspondence with  $\tambg(\nngh \su,\R)$ by showing that every morphism in $\tambg(\nngh\su,\R)$ determines and is determined by a morphism in $\tambh(\su,\ih\R)$.

 A morphism $\Omega$ in $\tambg(\nngh\su,\R)$ consists of a collection of ring homomorphisms $\{\Omega_P\co (\nngh\su)(G/P) \to \R(G/P) \text{ for all } P\le G\}$ that commute with the appropriate restriction, transfer, and norm maps.
Further, every element in $(\nngh \su)(G/G)$ is a sum consisting of elements in the image of the transfer map and sums of elements in the image of the norm map. (Indeed, every generator $N(s)$ in the image of the free summand $\Z\{\su(H/H)\}$  of $(\nngh \su)(G/G)$ is the norm of the element $s\otimes 1^{\otimes |G/H|-1}$ in $(\nngh \su)(G/H)$.) Thus,  the ring homomorphism $\Omega_G$ is completely determined by $\Omega_H$, and since $H$ is maximal in $G,$  the morphism $\Omega$ is completed determined by the collection of ring homomorphisms $$\{\Omega_{H'}\co \su^{\bx |G/H|}(H/H') \to \R(G/H') \text{ for all } H' \le H\}.$$

By Proposition~\ref{boxprodmaps},  the above collection of maps determines and is determined by a collection of Weyl equivariant maps $$\{\theta_{H'}\co \su(H/H')^{\otimes |G/H|} \to \R(G/H') \text{ for all } H' \le H\}$$ that satisfies the compatibility conditions given in Proposition~\ref{boxprodmaps}. 
But we can write every $\bm{s}^{\otimes |G/H|}$  in $\su(H/H')^{\otimes |G/H|}$ as a product over the $W_G(H')$-action. Thus, each $\theta_{H'}$ determines and is determined by a $W_G(H')$-equivariant homomorphism $\Lambda_{H'}\co \su(H/H') \to \R(G/H')$ because we can write $\theta_{H'}({\bm{s}}^{\otimes |G/H|})$ as the following product.
\begin{eqnarray*}
\theta_{H'}({\bm{s}}^{\otimes |G/H|}) &=& \theta_{H'}\left[\prod_{i=0}^{|G/H|-1} \g^i  \left(s_{\g^i}\otimes1^{\otimes|G/H|-1}\right)\right] \\
&=& \prod_{i=0}^{|G/H|-1} \g^i  \theta_{H'} \left(s_{\g^i}\otimes1^{\otimes|G/H|-1}\right) \\
&=& \prod_{i=0}^{|G/H|-1} \g^i  \Lambda_{H'}(s_{\g^i}).
\end{eqnarray*}

Lastly, we show that $\Omega$ is well-defined by showing that $\Omega_G$ sends the $TR$ submodule of $(\nngh\su)(G/G)$ to zero. Consider the element $N(a+b)-N(a)-N(b)-tr^G_H(g_H({\bm{a}},{\bm{b}}))$ in $TR$. Then $\Omega_G(N(a+b)-N(a)-N(b)-tr^G_H(g_H({\bm{a}},{\bm{b}})))$ equals 
\begin{eqnarray*}\lefteqn{ \ngh\Omega_H\left((a+b)\otimes 1^{\otimes |G/H|-1}\right) -} \\&&  \ngh\Omega_H\left(a\otimes 1^{\otimes |G/H|-1}\right) - \ngh\Omega_H\left(b\otimes 1^{\otimes |G/H|-1}\right)- tr^G_H\Omega_H\left(g_H({\bm{a}},{\bm{b}})\right),
\end{eqnarray*}
and we must show that this element is zero in $\R(G/G)$. First, 
$$\ngh\Omega_H\left((a+b)\otimes 1^{\otimes |G/H|-1}\right) = \ngh(\Lambda_H(a+b))
=\ngh \left(\Lambda_H(a) + \Lambda_H(b)\right).$$
Then since $\R$ is a Tambara functor, by Tambara reciprocity for sums, $$\ngh \left(\Lambda_H(a) + \Lambda_H(b)\right) - \ngh \left(\Lambda_H(a) \right) - \ngh \left( \Lambda_H(b)\right) - tr^G_H(g_H(\Lambda_H(a),\Lambda_H(b)))=0.$$ Thus, it remains to show that $$tr^G_H\left(\Omega_H\left(g_H({\bm{a}},{\bm{b}})\right)\right) = tr^G_H\left(g_H\left(\Lambda_H(a),\Lambda_H(b)\right)\right).$$

But, $g_H(\ba,\bb) = \sum_{f \in \mathcal I_k} \left(\bigotimes_{i=0}^{|G/H|-1} f(\g^i H)\right)$, and since $\nngh \su$ is a Tambara functor (and thus $(\nngh\su)(G/H)$ has a multiplication), it follows that $$\botimes_{i=0}^{|G/H|-1} f(\g^i H) = \prod_{i=0}^{|G/H|-1} \g^i \left(f(\g^iH)\otimes 1^{\otimes |G/H|-1}\right).$$ Thus,
\begin{eqnarray*}
tr^G_H\Omega_H\left(g_H(\ba,\bb)\right)&=& tr^G_H\Omega_H \left(\sum_{f \in \mathcal I_k/G}\left(\prod_{i=0}^{|G/H|-1} \g^i \left(f(\g^iH)\otimes 1^{\otimes |G/H|-1}\right)\right)\right) \\
&=& tr^G_H\left(\sum_{f \in \mathcal I_k/G}\left(\prod_{i=0}^{|G/H|-1}\g^i \Lambda_H(f(\g^i H))\right)\right) \\
&=& tr^G_H\left(g_H(\Lambda_H(a),\Lambda_H(b))\right).
\end{eqnarray*}

Next we show that $\Omega_G\left(N(tr^H_{H'}(x)) - tr^G_{H'}(f({\bm{x}}))\right)=0$. First, $$\Omega_G\left(N(tr^H_{H'}(x)) - tr^G_{H'}(f({\bm{x}}))\right)= \ngh\Omega_H\left(tr^H_{H'}(x)\otimes 1^{\otimes |G/H|-1}\right) - tr^G_{H'}\Omega_{H'}(f({\bm{x}})),$$ and $$\ngh\Omega_H\left(tr^H_{H'}(x)\otimes 1^{\otimes |G/H|-1}\right) = \ngh \Lambda_H(tr^H_{H'}(x)) = \ngh tr^H_{H'}(\Lambda_{H'}(x)).$$ Since $\R$ is a Tambara functor it follows that $$\ngh tr^H_{H'}(\Lambda_{H'}(x)) - tr^G_{H'}(f(\Lambda_{H'}(x)))=0,$$ and so we need to show that $\Omega_{H'}(f({\bm{x}}))= f(\Lambda_{H'}(x))$.

As detailed in Corollary~\ref{cor:TRtransfers}, $f(\Lambda_{H'}(x)) = \sum_{s=1}^r \left(\prod_{i=0}^{|G/H|-1} \g^i \g^{m_{i,s}} \Lambda_{H'}(x)\right)$. Further, 
\begin{eqnarray*} 
\Omega_{H'}(f({\bm{x}}) )&=& \sum_{s=1}^r \theta_{H'}\left(\g^{m_{0,s}}x \otimes \g^{m_{1,s}} x \otimes \cdots \otimes \g^{m_{|G/H|-1,s}}x\right)_s \\
&=& \sum_{s=1}^r \bigg(\prod_{i-0}^{|G/H|-1}\g^i \Lambda_{H'}(\g^{m_{i,s}} x)\bigg).
\end{eqnarray*}
Since $\Lambda_{H'}$ is Weyl equivariant it follows that $\Omega_{H'}(f({\bm{x}}))= f(\Lambda_{H'}(x))$. Therefore, since $\Omega_G$ maps all elements of the Tambara reciprocity submodule of $(\nngh \su)(G/G)$ to zero, $\Omega$ is well defined.
\end{proof}

We now prove Proposition \ref{TFisGCM}.
\begin{proof}[Proof of Proposition \ref{TFisGCM}] 
Let $\su$ be a $G$-Tambara functor. To show that $\su$ is a $G$-commutative monoid we first show that a map of orbits $G/H \to G/K$ induces a map $G/H\otimes \su \to G/K\otimes \su$. Consider the $K$-Tambara functor $i_K^*\su$.  By Lemma \ref{tfadj}  there is an adjunction between $\mathcal{N}^K_H$ and $\ih$, and hence a counit map $\mathcal{N}^K_H i_H^* i_K^*\su \to i_K^* \su$. To define $G/H\otimes \su \to G/K\otimes \su$ we apply $\mathcal{N}^G_K$ to the above counit map, which yields a map $\mathcal{N}^G_K \mathcal{N}^K_H i_H^* i_K^* \su \to \mathcal{N}^G_K i_K^* \su$. Since $\mathcal{N}^G_K\mathcal{N}^K_H$ is isomorphic to $\nngh$ and $i_H^*i_K^*$ is isomorphic to $\ih$ it follows that the above map is a map $\nngh \ih \su  \to \mathcal{N}^G_K i_K^* \su$.

If $X \to Y$ is a map of arbitrary $G$-sets we define the induced map $X \otimes \su \to Y \otimes \su$ as follows. First, we write $X$ and $Y$ as disjoint unions of orbits, so $X \cong \coprod_i G/H_i$ and $Y \cong \coprod_j G/K_j$. Thus, the map $X \to Y$ consists of a combination of fold maps, automorphisms, and canonical maps $G/H_i \to G/K_j$ where $H_i$ is a subgroup of $K_j$. Further, $$X \otimes \su \cong \Big(\coprod_i G/H_i \Big) \otimes \su \cong \bx_i (G/H_i \otimes \su) = \bx_i \mathcal{N}^G_{H_i} i_{H_i}^* \su,$$ and similarly, $Y \otimes \su$ is isomorphic to $\bx_j \mathcal{N}^G_{K_j} i_{K_j}^* \su$.

Thus, because the box product is the coproduct in $\tambg$, the induced map $X \otimes \su \to Y \otimes \su$ is a map $$\bx_i \mathcal{N}^G_{H_i} i_{H_i}^* \su \to \bx_j \mathcal{N}^G_{K_j} i_{K_j}^* \su$$ that is a combination of multiplication maps, Weyl actions, and maps $\mathcal{N}^G_{H_i}i_{H_i}^*\su \to \mathcal{N}^G_{K_j}i_{K_j}^*\su$ induced from the counit map.
\end{proof}

It remains the prove the following statement.

\begin{prop}\label{gcmtotf} If $\mackg$ has the $G$-symmetric monoidal structure of Theorem \ref{mainthm} and  $\M$ is a $G$-commutative monoid in $\mackg$, then $\M$ is a $G$-Tambara functor.
\end{prop}


 We show that if a Mackey functor $\M$ is a $G$-commutative monoid then it is a commutative Green functor using the basic properties of the functor $(-)\otimes (-)$. Defining the internal norm maps $N^K_H$ for all $H < K \le G$ that make $\M$ into a Tambara functor takes more work. However, we will only define the norm map $\ngh$ since any norm map $N^K_H$ in a $G$-Tambara functor $\su$ must agree with the analogous norm map in $i_K^*\su$.

 To define the norm $\ngh$ we first recognize that the Mackey functor $\ih\M$ is an $H$-Mackey functor that maintains the Weyl action defined on $\M$. Hence, $\ngh\ih\M$ should also remember this Weyl action, and in particular, for all subgroups $H'$ of $H$, $(\ih\M)(H/H')$ is isomorphic to $\M(G/H')$. Then since $\M$ is a $G$-commutative monoid the map of orbits $G/H \to G/G$ induces a map $\ngh\ih\M \to \M$. Therefore, we define the norm map $\ngh$ to be the composition below where $N$ is the map from Remark \ref{DefofN}. $$\M(G/H) \xrightarrow{N} (\ngh\ih\M)(G/G) \xrightarrow{} \M(G/G)$$ 
 
 In the proof of Proposition \ref{gcmtotf} we show that this composition satisfies all of the properties of a norm in a Tambara functor. But, in order to show that this composition specifically satisfies Property 6 of Definition \ref{axiomatictfdef} we first need to redefine $\ngh\ih\M$ with an alternate Weyl action and restriction map and then show that this new Mackey functor is isomorphic to the original. Specifically, to show Property 6 of Definition \ref{axiomatictfdef} we use the fact that $\ih\M$ maintains the Weyl action on $\M$, and hence there is a Weyl action on $\ngh\ih\M$ that combines this action with the original action.
 Moreover, since the restriction must map into the Weyl fixed points, using this alternate Weyl action requires us to tweak the restriction map as well.

Before stating this new Weyl action we provide further details regarding the original Weyl action and restriction maps of $\ngh\M$.

\begin{fact}[Explicit description of the Weyl action on $\ngh\M$]\label{nghweyl}  The Weyl action on $\ngh\M$ is as follows. When $H'\le H$, $(\ngh\M)(G/H')$ is $ \M^{\bx |G/H|}(H/H')$, and thus, the generator $\g$ of $W_G(H')$ acts on a simple tensor by  
\begin{eqnarray*}\label{gheqn}\small
 \g  (\bfm^{\otimes |G/H|}) &=& \left(\g^{|G/H|}  m_{\g^{|G/H|-1}}\right) \otimes m_e \otimes m_{\g} \otimes \cdots \otimes m_{\g^{|G/H|-2}},
 \end{eqnarray*} where $\g^{|G/H|}$ generates $W_H(H')$.
 
For all subgroups $K$ such that $H<K \le G$, the generator $\g$ of $W_G(K)$ acts on the generator $N(\bfm^{\times|G/K|})$ in $(\ngh\M)(G/K)$ by cyclically permuting the factors of $\bfm^{\times|G/K|}$ in the following way:
\begin{eqnarray*}
 \g  N\left(\bfm^{\times|G/K|}\right) &=& N(\g  (m_e\times m_{\g} \times \cdots \times m_{\g^{|G/K|-1}})) \\
  &=&  N( m_{\g^{|G/K|-1}} \times m_{e} \times m_{\g} \times \cdots \times m_{\g^{|G/K|-2}}).
 \end{eqnarray*}\end{fact}
 
\begin{fact}[Explicit description of the restriction maps of $\ngh\M$]\label{resdef} Recall that $H=C_{p^k}$, and let $K=C_{p^j}$ be a subgroup of $G$. If $j\le k$ then $res^j_{j-1}$ is the box product restriction of Definition~\ref{bxproddef}. If $j>k$ then for $tr(x)$ in $Im(tr^j_{j-1})$, $$res^j_{j-1}(tr(x)) = \sum_{g \in W_K(C_{p^{j-1}})} g  x,$$ and for a generator $N(\bfm^{\times |G/K|})$ in the image of the free summand in $(\ngh\M)(G/K)$,  $$res^j_{j-1} (N(\bfm^{\times|G/K|})) = \begin{cases} N\left(\prod_{i=0}^{p-1}\bfm^{\times|G/K|}\right) & \text{ if } j-1>k \\ \botimes_{i=0}^{p-1}\bfm^{\otimes |G/K|} & \text{ if } j-1=k \end{cases},$$ where  $\prod_{i=0}^{p-1} \left(\bfm^{\times |G/K|}\right)$ is the $p$-fold Cartesian product of $\bfm^{\times|G/K|}$ and $\bigotimes_{i=0}^{p-1} \bfm^{\otimes |G/K|}$ is the analogous tensor product.
\end{fact}

\begin{prop}\label{newaction} Let $\M$ be a $G$-Mackey functor. There is a Weyl action on $\ngh\ih\M$ that combines the Weyl action given in Fact \ref{nghweyl} with the Weyl action defined on $\M$. \end{prop}

\begin{proof}
We define the Weyl action on $\ngh\ih\M$ as follows.
 
\begin{itemize}
\item For all subgroups $H'$ of $H$, the generator $\g$ of $W_G(H')$ acts on a simple tensor of $(\ngh\ih\M)(G/H')$ by  cyclically permuting the factors of the tensor product {\emph{and}} by acting on each factor. Thus, \begin{eqnarray*}\label{niact1} \g \left(\bfm^{\otimes |G/H|} \right)& =&   \g m_{\g^{|G/H|-1}} \otimes \g m_e \otimes \g m_{\g} \otimes \cdots \otimes \g m_{\g^{|G/H|-2}}. \hspace{30pt}  \end{eqnarray*} 
\item Let $r = |G/H|-|G/K|+1$. For all subgroups $K$ such that $H<K\le G$ the generator $\g$ of $W_G(K)$ acts on a generator $N\left(\bfm^{\times|G/K|}\right)$ of the image of the free summand of $(\ngh\ih\M)(G/K)$ by \begin{eqnarray*}\label{niact2} \g N\left(\bfm^{\times |G/K|} \right) &=& N\left(\g^r m_{\g^{|G/K|-1}} \times \g m_e \times \g m_{\g} \times \cdots \times \g m_{\g^{|G/K|-2}}\right). \hspace{30pt} \end{eqnarray*}  
\end{itemize}

Further, the restriction maps of $\ngh\ih\M$ must remain compatible with the Weyl action. So, let $\left( g \bfm\right)^{\otimes |G/H|}$ (or $\left(g\bfm\right)^{\times |G/K|}$) denote $g$ acting on each factor of the  product: $$\left(g \bfm\right)^{\otimes |G/H|} = g m_e \otimes g m_{\g} \otimes \cdots \otimes g m_{\g^{|G/H|-1}}.$$  If $K=C_{p^j}$ is a subgroup such that $H<K\le G$, then  $$res^j_{j-1} (N(\bfm^{\times|G/K|})) = \begin{cases} N\left(\prod\limits_{i=0}^{p-1}\left(\g^{i|G/K|}\bfm\right)^{\times|G/K|}\right) & \text{ if } j-1>k \\ \botimes\limits_{i=0}^{p-1}\left(\g^{i|G/K|}\bfm\right)^{\otimes |G/K|} & \text{ if } j-1=k \end{cases}.$$
\end{proof}
 
 \begin{thm}\label{compositioniso}
Let $\M$ be a $G$-Mackey functor. The $G$-Mackey functor $\ngh\ih\M$ with the Weyl action from Proposition \ref{newaction} is isomorphic to $\ngh\ih\M$ with the original Weyl action described in  Fact \ref{nghweyl}. 
\end{thm}

\begin{proof}
Let $U\ih\M$ denote the underlying $H$-Mackey functor of $\ih\M$. So, $U\ih\M$ does not remember the Weyl action from $\M$. Hence, we can let $\ngh U\ih\M$ denote $\ngh\ih\M$ with the Weyl action as defined in Definition \ref{nghweyl}. We will define an isomorphism $\chi\co \ngh U\ih\M \to \ngh\ih\M$ by defining a collection of isomorphisms $$\{\chi_P\co (\ngh U\ih\M)(G/P) \to (\ngh\ih\M)(G/P) \text{ for all } P\le G\}.$$ First, if $H' \le H$, then define $$\chi_{H'}\co   (U\ih\M)^{\bx|G/H|}(H/H') \to (\ih \M)^{\bx|G/H|}(G/H')$$ to be   $1\otimes \g \otimes \cdots \otimes \g^{|G/H|-1}$ on the tensor summand. On the image of the transfer map define $\chi_{H'}$ so that the appropriate diagram commutes.  Similarly, for subgroups $K$ such that $H<K\le G$ let $\chi_K$ be $1 \times \g \times \cdots \times \g^{|G/K|-1}$ on the image of the free summand and on the image of the transfer require that the appropriate diagram commutes. 
\end{proof}

Finally, we complete the proof of Theorem \ref{mainthm2} by proving Proposition \ref{gcmtotf}.

\begin{proof}[Proof of Proposition \ref{gcmtotf}] Since $\M$ is a $G$-commutative monoid,  $(-) \otimes \M$ extends to a functor $\set \to \mackg$. We will first show that $\M$ is a commutative $G$-Green functor by showing that $\M$ satisfies the categorical definition of a Green functor as given in \cite{LewisGF} or \cite{Ventura}. We will then show that the codomain of $(-)\otimes \M$ is $\greeng$. We need the latter fact so that the internal norm maps that we will define to make $\M$ into a Tambara functor are multiplicative.

 Let $*$ be the orbit $G/G$ in $\set$.  The projection map $p\co * \amalg * \to *$ induces a multiplication map $\M \bx \M \to \M$ on $\M$, and the  inclusion map $i\co \emptyset \hookrightarrow *$ induces a unit map $\A \to \M$.  Applying   $(-) \otimes \M$ to the following three diagrams in $\set$ results in the  commutative diagrams in $\mackg$ needed to make $\M$ a $G$-Green functor.
 $$\xymatrix@=1pc{{*}\amalg{*} \amalg{*} \ar[rr]^{id \amalg p} \ar[d]_{p \amalg id} && {*} \amalg {*} \ar[d]^p \\ {*}\amalg{*} \ar[rr]^p && {*}} $$ $$\xymatrix@=1pc{\emptyset \amalg {*} \ar[rr]^{i \amalg id} \ar[rrd]_{=}& & {*} \amalg {*} && \emptyset \amalg{*} \ar[ll]_{i \amalg id} \ar[lld]^{=} \\  & & {*}} \hspace{.5in} \xymatrix@=1pc{{*} \amalg {*} \ar[rr]^{\tau} \ar[rd]_p & & {*} \amalg {*} \ar[ld]^p \\ & {*} & }$$ To show that the codomain of $(-)\otimes \M$ is $\greeng$ we note that $G/H\otimes \M$ is a commutative Green functor for all orbits of $G$ because $\M$ is a commutative Green functor and  both functors $\ngh$ and $\ih$ are strong symmetric monoidal. Then given a map $f\co G/H \to G/K$ in $\set$ we show that the induced map $G/H \otimes \M \to G/K \otimes \M$ is a morphism in $\greeng$ by applying $(-)\otimes \M$ to the diagrams below. $$\xymatrix@=1pc{G/H \amalg G/H \ar[rr]^(.6){p} \ar[dd]_{f\amalg f}& &G/H \ar[dd]^f  \\ \\ G/K \amalg G/K \ar[rr]^(.6){p} & & G/K} \hspace{.5in} \xymatrix@=1pc{ \emptyset \ar@{^{(}->}[rr]^i \ar@{^{(}->}[ddrr]_i & & G/H \ar[dd]^f  \\  \\& & G/K}$$

It remains to define norm maps $N^K_H\co \M(G/H) \to \M(G/K)$ for all subgroups $H<K\le G$. However, we need only construct the norm maps $N^G_H$ since we can subsequently  build every $N^K_H$ by applying the process below to $\ik\M$. Let  $\pi^*\co N^G_Hi_H^*\M \to \M$ be the map induced from  $\pi\co G/H \to G/G$. Since $\ngh\ih\M$ has the Weyl action described in Proposition \ref{newaction}, $(\ih\M)(H/H)$ is isomorphic to $\M(G/H)$. Then let the map $N\co \M(G/H) \to (\ngh\ih\M)(G/G)$ be as given in Remark \ref{DefofN}, and define the norm map $\ngh$ by the composition
$$\M(G/H) \xrightarrow{N} (\ngh\ih\M)(G/G) \xrightarrow{\pi^*_G} \M(G/G).$$  Since the functor $\ngh \co \mackh \to \mackg$ satisfies Property (a) of the Main Theorem, the above composition satisfies Property 3 of Definition \ref{axiomatictfdef}. The composition satisfies Tambara reciprocity by the construction of the functor $\ngh$. 

Next we show that the norm map $\pi^*_GN$ factors through the Weyl action (i.e. that $\pi^*_GN$ satisfies Property 5 of Definition \ref{axiomatictfdef}). The Weyl action on $G/H \otimes \M$ is induced from automorphisms of $G/H$, which are given by multiplication by $\g^j$ for some $\g^j$ in $W_G(H)$. Hence, the commutative diagram of $G$-sets on the left below  induces the commutative diagram of Mackey functors on the right. $$\xymatrix{ G/H \ar[r]^{\g^j\cdot} \ar[dr]_{\pi} & G/H \ar[d]^{\pi} \\ & G/G} \hspace{1in} \xymatrix{ G/H \otimes \M \ar[r]^{(\g^j\cdot)^*} \ar[dr]_{\pi^*} & G/H\otimes \M \ar[d]^{\pi^*} \\ & \M}$$ It follows that $\pi^*_G(N(\g^jx)) = \pi^*_G(N(x))$ for all $x$ in $\M(G/H)$.

 Finally, we show that $res^G_H\pi^*_GN(x) = \prod_{\g^j\in W_G(H)} \g^j x$ for all $x$ in $\M(G/H)$. By Proposition \ref{newaction} and properties of morphisms of Mackey functors we have
 $$res^G_H \pi^*_GN(a) 
= \pi^*_H res^G_H N(a) 
=\pi^*_H\left(a \otimes \g a\otimes \cdots \otimes \g^{|G/H|-1} a \right).$$  Since the $G$-symmetric monoidal structure is compatible with the forgetful functor $i_H\co \set \to \mathscr{S}et^{Fin}_H$, it follows that $\pi^*_H$ is induced from $i_H \pi$, which is the fold map $\coprod_{|G/H|} H/H \to H/H$.  Therefore, $$\pi_H^*\left(a\otimes \g a\otimes \cdots \otimes \g^{|G/H|-1}\right) = a\g a \cdots \g^{|G/H|-1} a. \qedhere$$ \end{proof}

\section*{Acknowledgments} The authors were supported by NSF Grant DMS--1207774. The authors thank Kate Ponto for her help and suggestions in organizing and writing this paper. They  also thank the referee for their meticulous reading of this paper and their numerous valuable recommendations.
 
\bibliographystyle{alpha}
\bibliography{ref}

\end{document}